\documentclass[11pt,reqno]{amsart}
\usepackage[utf8]{inputenc}
\usepackage{amsmath,amssymb,amsthm,mathrsfs, color}
\usepackage{a4wide}
\usepackage{graphics,subfigure}
\usepackage{epsfig,float}

\usepackage[colorlinks=true]{hyperref}
\hypersetup{%
    ,urlcolor=blue
    ,citecolor=red
    ,linkcolor=blue
    }

\newtheorem{theorem}{Theorem}[section]
\newtheorem{prop}{Proposition}[section]
\newtheorem{lemma}{Lemma}[section]
\newtheorem{corollary}{Corollary}[section]

\newtheorem{remark}{Remark}[section]

\newcommand{\cM}{\mathcal{M}}

\newcommand{\R}{{\mathbb R}}
\newcommand{\N}{{\mathbb N}}
\newcommand{\E}{{\mathbb E}}

\def\tn{\tilde n}

\makeatletter
 \@addtoreset{equation}{section}
 \makeatother
 
\allowdisplaybreaks
\usepackage{appendix}
\usepackage{hyperref}

\title[Approximation of Stochastic Volterra Equations]{Approximation of Stochastic Volterra Equations with kernels of  completely monotone type}
\date{\today}
\author{Aur\'elien Alfonsi}
\address{Aur\'elien Alfonsi, CERMICS, Ecole des Ponts, Marne-la-Vall\'ee, France. MathRisk, Inria, Paris, France.}
\email{aurelien.alfonsi@enpc.fr}

\author{Ahmed Kebaier}
\thanks{This work benefited from the support of the ``chaire Risques financiers'', Fondation du Risque.}
\address{Ahmed Kebaier, Laboratoire de Math\'ematiques et Mod\'elisation d'\'Evry, CNRS, Univ Evry, Universit\'e Paris-Saclay, 91037, Evry, France}
\email{ahmed.kebaier@univ-evry.fr}

\subjclass[2010]{60H35 60G22 91G60 45D05}

\keywords{Stochastic Volterra Equation, Euler scheme, Strong error, Fractional kernel, Rough volatility models}
\begin{document}


\begin{abstract}
  In this work, we develop a multifactor approximation for $d$-dimensional Stochastic Volterra Equations (SVE) with Lipschitz coefficients and kernels of completely monotone type that may be singular. First, we prove an $L^2$-estimation between two SVEs with different kernels, which provides a quantification of the error between the SVE and any multifactor Stochastic Differential Equation (SDE) approximation. For the particular rough kernel case with Hurst parameter lying in $(0,1/2)$, we propose various approximating multifactor kernels, state their rates of convergence and illustrate their efficiency for the rough Bergomi model. Second, 
  we study a Euler discretization of the multifactor SDE and establish a convergence result towards the SVE that is uniform with respect to the approximating multifactor kernels. These obtained results lead us to build a new multifactor Euler scheme that reduces significantly the computational cost in an asymptotic way compared to the Euler scheme for SVEs.  Finally, we show that our multifactor Euler scheme outperforms the Euler scheme for SVEs for option pricing in the rough Heston model.  
\end{abstract}

\maketitle
\section{Introduction}
In recent years, there has been significant and growing interest in studying Stochastic Volterra Equations (SVE) since they arise in many applications such as mathematical finance, biology, physics, and engineering. Several studies have investigated the SVE under regular kernels, see e.g.\ Berger and Mizel \cite{BerMiz1,BerMiz2}, Protter \cite{Prot}, Pardoux and Protter \cite{ParPro}, and under non-regular kernels as well, see e.g.\ Cochran et al.\ \cite{CocLeePot}, 
Coutin and Decreusefond \cite{CouDec}, Decreusefond \cite{Dec}, Wang \cite{Wan}, Zhang \cite{Zhang},  and the references therein. More recently, much attention in quantitative finance has centered on using the SVE with a fractional  kernel having a small Hurst parameter $H\simeq 0.1$ to reproduce several  statistical stylized facts observed on real markets 
such as the path roughness of the volatility shown by   Gatheral et al.\ \cite{GatJaiRos} or the  pronounced smile of the implicit volatility curve occurring for very short time maturities (see e.g.\  Fukusawa \cite{Fuku1,Fuku2}, Bayer et al.\ \cite{BayFriGulHorSte, BaFrGa} and Friz et al. \cite{friz}). 
From a practical point of view, on the one hand Zhang \cite{Zhang_euler}  proposed to approximate the SVE with possibly singular kernels and globally Lipschitz  coefficients using an Euler discretization scheme. More recently, Richard et al.~\cite{RTY} have updated the study for the Euler scheme and proposed a Milstein discretization scheme  improving the rate of strong convergence.  On the other hand, inspired by the works of Carmona and Coutin \cite{CarCou1,CarCou2},  Harms and Stefanovits \cite{HarStef}, Abi Jaber and El Euch  \cite{EleAbi} proposed a multifactor  approximation scheme for the rough Heston model. Unlike the Euler approximation, the multifactor scheme that we develop in this paper for kernels of completely monotone type that may be singular, features a Markovian structure that allows the use of a wide range  of  usual techniques available in the literature, namely the Euler scheme for stochastic differential equations and  higher order schemes for the weak error (see e.g.\   Talay and Tubaro \cite{TaTu}, Kusuoka~\cite{Kusuoka}, Ninomiya and Victoir~\cite{NiVi}, Alfonsi~\cite{Alfonsi} and Shinozaki~\cite{Shinozaki}), Multilevel Monte Carlo methods (see e.g.\ Giles \cite{Gil}, Ben Alaya and Kebaier \cite{BAK}, Lemaire and Pagès \cite{LemPag1}), the variance reduction techniques (see e.g.\  Newton~\cite{Newton}, Jourdain and Lelong~\cite{JoLe},  Lemaire and Pagès \cite{LemPag2}, Belomestny et al.~\cite{BHNU}) etc.  This gives more flexibility for the approximation setting. 
\\\\
In this paper, we are interested in approximating the SVE in a general form given by 
\begin{equation}\label{SEV:intro}
X_t=x_0+\int_0^t G_1(t-s)b(X_s)ds+\int_0^t G_2(t-s)\sigma(X_s)dW_s, t\ge 0,
\end{equation}
where $x_0 \in \R^d$, $b : \R^d \to \R^d$, $\sigma : \R^d \to \mathcal{M}_d(\R)$ are globally Lipschitz continuous coefficients, $W$ is a standard Brownian motion in $\R^d$ and $G_1, G_2: \R^*_+\to  \mathcal{M}_d(\R)$ are kernels of the form 
\begin{equation}\label{kernel_form}
  G_j(t)= \int_{\R_+} e^{-\rho t} M_j(\rho) \,\lambda(d\rho), \quad\text{ for } t\in]0,+\infty[,
\end{equation}
with
bounded measurable functions $M_1,M_2:\R_+ \to \mathcal{M}_d(\R)$ and a measure~$\lambda$ on $\R_+$ satisfying  $\int_{\R_+} e^{-\rho t}  \,\lambda(d\rho) <+\infty$. Note that when $M_1$ and $M_2$ are non-negative scalar functions, $G_1$ and $G_2$ are known in the literature as completely monotone kernels. In particular, the singular fractional kernel  with Hurst parameter that lies in $(0,1/2)$ is covered within this framework.  More precisely, we approximate the solution to \eqref{SEV:intro} by a multifactor approximation that corresponds to a stochastic differential equation in a higher dimension. We prove a strong convergence error for our multifactor approximation scheme that holds in this general $d$-dimensional setting. To do so, we proceed in two steps:   first we truncate the integrals defining $G_j$ and second we discretize  the measure $\lambda$ on the truncated interval $[0,K]$.  We denote respectively by $X^K$ and $\hat X^K$ the corresponding SVE processes. Thus, in Section~\ref{Sec_Strong} we derive a first strong convergence on the error between the processes  $X$ and its truncated version $X^K$  and a second one for the error between $X^K$ and $\hat X^K$. We also obtain a general non asymptotic result on the approximation of~\eqref{SEV:intro} given any approximation of $G_1$ and $G_2$, see Theorem~\ref{thm_non_asymp}. Though being a natural question, it seems to us that such a result has not been yet stated in the literature. The next section is dedicated to study the multifactor approximation approach when combined with a Euler scheme on te regular grid with $N$~time steps. On the one hand, we analyse the error between two Euler schemes with different kernels, see Theorem~\ref{thm_non_asymp_scheme}. This gives, combined with the recent convergence results of Richard et al.~\cite{RTY}, a uniform convergence to the SVE with respect to the approximating kernels, see Corollary~\ref{cor_euler}. Interestingly, it turns out from our strong error analysis that for the rough kernel with $H\simeq 0.1$, there is no need to have a very accurate approximation of the kernel to run a multifactor Euler scheme since the main error comes from the discretization. On the other hand, we show in Theorem~\ref{prop_euler} that the Euler scheme on the SVE~\eqref{SEV:intro} for kernels of completely monotone type coincides with the Euler scheme on the corresponding multifactor SDE. Thus, by approximating kernels~\eqref{kernel_form} by a finite combination of~$n$ ($n<<N$) exponentials, we can reduce the computational cost from $N^2$ to $n\times N$. These new results hold for general $d$-dimensional Stochastic Volterra Equations and any multifactor approximation provided that it is accurate enough.  
Then, Section~\ref{Sec_Rough} is devoted to  the  study of the rough kernel, where we propose various procedures to obtain approximating kernels and give their precise rates of convergence.  We illustrate in Section~\ref{Sec_Num} our theoretical results on the kernel approximation and give a first financial application with the celebrated rough Bergomi model. In a second financial application on pricing options in the rough Heston model, based on our theoretical results of Section~\ref{Sec_Euler}, we provide a new scheme with a reduced computational cost that has the same accuracy as the Euler scheme for SVEs.  Then, we combine the kernel approximation and the Euler scheme to illustrate the efficiency of our approach to calculate option prices in the rough Heston model. It is worth stressing that the gain with respect to the Euler scheme for SVEs is important: as an example, for a precision of order $10^{-3}$, we get a computational time $5$~times smaller, see Table~\ref{table_1}.  We also compare on the same example the very recent hybrid multifactor scheme by~R{\o}mer~\cite{Romer} to our scheme and show it has a better performance on a benchmark case for pricing a European type options.

\section{General Framework and preliminary results}\label{Sec_GF}

We consider the SVE in a general form given by 
\begin{equation}\label{SVE}
X_t=x_0+\int_0^t G_1(t-s)b(X_s)ds+\int_0^t G_2(t-s)\sigma(X_s)dW_s, t\ge 0,
\end{equation}
where $x_0 \in \R^d$, $b : \R^d \to \R^d$, $\sigma : \R^d \to \mathcal{M}_d(\R)$ are globally Lipschitz continuous coefficients i.e.
\begin{equation} \label{cd:lip}\exists L>0, \forall x,y \in \R^d, |b(x)-b(y)|+\|\sigma(x)-\sigma(y)\|\le L|x-y|,
\end{equation}
$W$ is a standard Brownian motion in $\R^d$ and $G_1, G_2: \R^*_+\to  \mathcal{M}_d(\R)$ are kernels that satisfy
\begin{equation}\label{cond_kernel_SVE}
\int_0^T \left( \|G_1(s)\|+\|G_2(s)\|^2 \right)ds <\infty,  \text{ for every } T\in\mathbb R_+.
\end{equation}
Then, we can apply Theorem~3.1~\cite{Zhang} and get that there exists a unique strong solution to~\eqref{SVE}. Note that if $\int_0^T \left(\|G_1(s)\|+\|G_2(s)\|^2 \right)ds<\infty$ for some $T>0$, then there exists a unique strong solution $(X_t,t\in[0,T])$ up to time~$T$. Obviously, those conditions do not depend on the choice of the norms $|\cdot|$ and $\|\cdot\|$ on $\R^d$ and $\mathcal{M}_d(\R)$. In this paper, we will use the Euclidean norm on~$\R^d$ and the Frobenius norm on~$\mathcal{M}_d(\R)$, and we recall that we have
\begin{equation}\label{Frob_prop}
  \forall A,B \in \mathcal{M}_d(\R), \forall x\in \R^d,\  \|AB\|\le\|A\|\|B\| \text{ and } |Ax|\le \|A\||x|.
\end{equation}

In this paper, we are interested in the approximation of~\eqref{SVE} when there exists bounded measurable functions $M_1,M_2:\R_+ \to \mathcal{M}_d(\R)$ and a measure~$\lambda$ on $\R_+$ satisfying 
\begin{equation}\label{cond_lambda}\forall t>0, \bar{G}(t)= \int_{\R_+} e^{-\rho t}  \,\lambda(d\rho) <+\infty,
\end{equation}
such that 
\begin{equation}\label{def_kernels}
  G_j(t)= \int_{\R_+} e^{-\rho t} M_j(\rho) \,\lambda(d\rho), \quad\text{ for } t\in]0,+\infty[. 
\end{equation}
We note ${\bf M}_j=\sup_{\rho \ge 0}\|M_j(\rho)\|$ and trivially have $\|G_j(t)\|\le {\bf M}_j\bar{G}(t)$. We will assume through the paper that $\bar{G}\in L^2_{\rm loc}(\R_+^*,\R_+)$, i.e.
\begin{equation}\label{cond_Gbar}
  \forall T>0, \int_0^T \bar{G}(t)^2 dt < \infty,
\end{equation}
and therefore condition~\eqref{cond_kernel_SVE} is satisfied. 

    In the one-dimensional case, the kernel $G_j$ is completely monotone when $M_i\ge 0$ by the celebrated Hausdorff--Bernstein--Widder theorem \cite[Theorem IV.12b]{Wid}. In this paper, we will be particularly interested by the rough kernel \begin{equation}\label{def_rough_kernel}G_{\lambda_H}(t)=\frac{t^{H-1/2}}{\Gamma(H+1/2)},
    \end{equation}
    with parameter $H\in(0,1/2)$. It satisfies $G_{\lambda_H}(t)=\int_{\R_+} e^{-\rho t}  \,\lambda_H(d\rho)$
with
\begin{equation}\label{frac_mes}\lambda_H(d\rho)=c_H{\rho^{-H-1/2}}d\rho,  \text{ and } c_H:=\frac{1}{\Gamma(H+1/2)\Gamma(1/2-H)}.
\end{equation}
The principle of the approximation is rather simple. We approximate the measure~$\lambda$ by a finite discrete measure. Then, the next proposition ensures that the Stochastic Volterra Equation~\eqref{SVE} can be obtained from the solution of a classical SDE, for which many numerical methods have been developed. Thus, the goal of the paper is to analyze the error made when replacing the measure~$\lambda$ by a finite discrete measure. We will focus in this paper on strong error estimates.
\begin{prop}\label{prop:SVE-SDE} Let us assume that $\lambda(d\rho)=\sum_{i=1}^n \alpha_i \delta_{\rho_i}(d\rho)$  with $\alpha_i\ge 0$ and $\rho_1<\dots<\rho_n$.
  \begin{enumerate}
  \item Let us assume $M_1=M_2=M$ and $\textup{rank}([\alpha_1M(\rho_1) \dots \alpha_nM(\rho_n)])=d$ so that there exist $x_0^1,\dots,x_0^n \in \R^d$ such that $\sum_{i=1}^n \alpha_i M(\rho_i) x_0^i =  x_0$. Then, the solution of~\eqref{SVE} is given by $\sum_{i=1}^n \alpha_i M(\rho_i)X_t^{\rho_i}$, where $(X_t^{\rho_1},\dots,X_t^{\rho_n})$ is the solution of the  $(n\times d)$-dimensional Stochastic Differential Equation defined by
    \begin{equation}\label{SDE_approx_1}X_t^{\rho_i}=x_0^i - \int_0^t \rho_i (X_s^{\rho_i}-x_0^i)ds + \int_0^t b\left(\sum_{j=1}^n \alpha_j M(\rho_j)X_s^{\rho_j}\right)ds +  \int_0^t \sigma \left(\sum_{j=1}^n \alpha_j M(\rho_j)X_s^{\rho_j}\right)dW_s. 
    \end{equation}
  \item Let us assume  $\textup{rank}([\alpha_1M_1(\rho_1) \dots \alpha_n M_1(\rho_n) \ \alpha_1M_2(\rho_1) \dots \alpha_n M_2(\rho_n) ])=d$ so that there exist $x_0^1,\dots,x_0^n,y_0^1,\dots,y_0^n \in \R^d$ such that $\sum_{i=1}^n \alpha_i [ M_1(\rho_i) x_0^i+M_2(\rho_i) y_0^i] =  x_0$. Then, the solution of~\eqref{SVE} is given by $X_t=\sum_{i=1}^n \alpha_i M_1(\rho_i)X_t^{\rho_i}+ \sum_{i=1}^n \alpha_i M_2(\rho_i)Y_t^{\rho_i}$, where $(X_t^{\rho_1},Y_t^{\rho_1},\dots,X_t^{\rho_n},Y_t^{\rho_n})$ is the solution of the  $(2n\times d)$-dimensional Stochastic Differential Equation defined by
    \begin{align}
      X_t^{\rho_i}&=x_0^i - \int_0^t \rho_i (X_s^{\rho_i}-x_0^i)ds + \int_0^t b\left(\sum_{j=1}^n \alpha_j M_1(\rho_j)X_s^{\rho_j}+\sum_{j=1}^n \alpha_j M_2(\rho_j)Y_s^{\rho_j}\right)ds, \notag \\
      Y_t^{\rho_i}&=y_0^i - \int_0^t \rho_i (Y_s^{\rho_i}-y_0^i)ds +  \int_0^t \sigma \left(\sum_{j=1}^n \alpha_j M_1(\rho_j)X_s^{\rho_j}+\sum_{j=1}^n \alpha_j M_2(\rho_j)Y_s^{\rho_j}\right)dW_s. \label{SDE_approx_2}
    \end{align}
  \end{enumerate}
\end{prop}
\begin{proof}
  Let us first consider the case $M_1=M_2=M$. The SDE~\eqref{SDE_approx_1} has Lipschitz coefficients and therefore has a unique strong solution. Since $d\left(e^{\rho_i t}(X_t^{\rho_i}-x_0^i)  \right) =e^{\rho_i t}b\left(\sum_{j=1}^n \alpha_j M(\rho_j)X_t^{\rho_j}\right)dt + e^{\rho_i t}  \sigma \left(\sum_{j=1}^n \alpha_j M(\rho_j)X_s^{\rho_j}\right)dW_t$, we get
  $$X_t^{\rho_i}=x_0^i + \int_0^t e^{-\rho_i (t-s)} b\left(\sum_{j=1}^n \alpha_j M(\rho_j)X_s^{\rho_j}\right)ds +  \int_0^t  e^{-\rho_i (t-s)} \sigma \left(\sum_{j=1}^n \alpha_j M(\rho_j)X_s^{\rho_j}\right)dW_s. $$
  We left multiply this equation by $\alpha_i M(\rho_i)$ and then sum over~$i$ to obtain that $\sum_{i=1}^n \alpha_i M(\rho_i)X_t^{\rho_i}$ solves~\eqref{SVE}. The strong uniqueness result (Theorem 3.1~\cite{Zhang}) gives the claim. 

  In the general case, we similarly get
  \begin{align*}
    X_t^{\rho_i}&=x_0^i + \int_0^t e^{-\rho_i (t-s)} b\left(\sum_{j=1}^n \alpha_j M_1(\rho_j)X_s^{\rho_j}+\sum_{j=1}^n \alpha_j M_2(\rho_j)Y_s^{\rho_j}\right)ds,  \\
      Y_t^{\rho_i}&=y_0^i +   \int_0^t  e^{-\rho_i (t-s)} \sigma \left(\sum_{j=1}^n \alpha_j M_1(\rho_j)X_s^{\rho_j}+\sum_{j=1}^n \alpha_j M_2(\rho_j)Y_s^{\rho_j}\right) dW_s. 
  \end{align*}
We then left multiply the first equation by $\alpha_i M_1(\rho_i)$ and the second equation by $\alpha_i M_2(\rho_i)$, and sum over $i$ to get the claim. 
\end{proof}

\section{Strong error analysis for the approximation}\label{Sec_Strong}

To analyse the error between the SVE and its approximation by using kernels $\hat{G}_j(t)$ supported by a finite discrete measure (as in Proposition~\ref{prop:SVE-SDE}), we proceed in two steps. First, we analyse the truncation error when replacing the kernels by the kernels obtained by truncating the measure~$\lambda$ in~\eqref{def_kernels}. Second, we analyse the error between the SVE with the truncated kernels and the approximating kernels.

For any $K>0$, we introduce then the truncated convolution kernels $G_j^K:\R_+ \rightarrow \mathcal{M}_d(\R)$, $j\in \{1,2\}$, that are defined as follows:
\begin{equation}\label{def_truncated_kernels}
G^K_j(t)=\int_{[0,K)}e^{-\rho t} M_j(\rho)\lambda(d\rho), \quad \mbox{for all } t\ge 0.
\end{equation}
Thus, the kernel $G^K_j$ approximates the kernel $G_j$ defined by~\eqref{def_kernels} as $K\to +\infty$. Since ${\bf M}_j=\sup_{\rho\ge 0}\|M_j(\rho)\| <\infty$, we have the following uniform bound:
\begin{equation}\label{unifbound_GK}\forall K>0, \|G^K_j(t)\|\le {\bf M}_j \bar{G}(t) \text{ with } \bar{G}(t)=\int_{\R_+}e^{-\rho t} \lambda(d\rho).
\end{equation}

We introduce the stochastic convolution equation  $X^K$ associated to the kernels $G^K_j$, $j\in \{1,2\}$, given by
\begin{equation}\label{sde:K}
X^K_t=x_0+\int_0^t G_1^K(t-s)b(X^K_s)ds+\int_0^tG_2^K(t-s)\sigma(X^K_s)dW_s, \qquad x_0\in \R^d.
\end{equation}
We also consider for $c>0$, the resolvant of second kind ${\rm E}_c(t)$ that solves the equation \begin{equation}\label{resolvante}{\rm E}_c(t)= \bar{G}^2(t) +\int_0^t c \bar{G}^2(t-s){\rm E}_c(s)ds.
\end{equation}
Since $ \bar{G}^2 \in L^1_{\rm loc}(\R_+^*,\R_+)$ by~\eqref{cond_Gbar}, we get that ${\rm E}_c(t)$ is well defined and belongs also to  $L^1_{\rm loc}(\R_+^*,\R_+)$ (see Subsection A.3~\cite{EleAbi} and Theorem 2.3.1~\cite{GriLonStaf}).

\begin{prop}\label{prop:tight_r(K)} 
  Let $\lambda$ be a positive measure such that
  \begin{equation}\tag{H1}\label{eq_H1}\forall K>0, \ r(K):=\int_{[K,+\infty)} \int_{[K,+\infty)} \frac 1 {\rho_1+\rho_2} \lambda(d\rho_1) \lambda(d\rho_2) <\infty.
  \end{equation}
  Then, for any $T>0$,  there exists a positive constant $C$ that depends on $T$, $\lambda$, ${\bf M}_1$, ${\bf M}_2$, $L$, $|b(0)|$ and $\|\sigma(0)\|$ such that 
\begin{equation}\label{Truncation_estimate}
\forall t \in [0,T], \ \E\left[\big| X_t-X^{K}_t\big|^2\right] \leq C \times r(K).
\end{equation}
\end{prop}
\begin{proof} 
We note $\Delta^K_j(t)=G_j(t)-G^K_j(t)$.  We have for all $t\ge0$
\begin{eqnarray*}
|X_t-X^{K}_t|^2&\le& 4\bigg(
\bigg|\int_0^t  \Delta^K_1(t-s) b(X_s)ds\bigg|^2 +\bigg|\int_0^t   \Delta^K_2(t-s)  \sigma(X_s)dW_s\bigg|^2\\
&+&\bigg|\int_0^t G^K_1(t-s)\big[ b(X_s)-b(X^{K}_s)\big]ds\bigg|^2 + \bigg|\int_0^t G^K_2(t-s)\big[ \sigma(X_s)-\sigma(X^{K}_s)\big]dW_s\bigg|^2
\bigg),
\end{eqnarray*}
by using the inequality $(a+b+c+d)^2\le 4(a^2+b^2+c^2+d^2)$. Then, we get by using Jensen's inequality, the It\^o isometry and~\eqref{Frob_prop}:
\begin{align*}
&\E\left[ |X_t-X^{K}_t|^2\right]\le 4t \int_0^t  \|\Delta^K_1(t-s)\|^2 \E[|b(X_s)|]^2ds + 4 \int_0^t  \|\Delta^K_2(t-s)\|^2 \E[\|\sigma(X_s)\|^2]ds\\
&+4t \int_0^t \|G^{K}_1(t-s)\|^2 \E[ \big| b(X_s)-b(X^{K}_s)\big|^2]ds + 4 \int_0^t \|G^{K}_2(t-s)\|^2\E[\| \sigma(X_s)-\sigma(X^{K}_s)\|^2]ds.
\end{align*}
Then, by using the Lipschitz property~\eqref{cd:lip}  we get for $c_1:=8(T\vee 1)\big(|b(0)|^2\vee\|\sigma(0)\|^2 +L^2 \sup_{t\in [0,T]}\E[|X_t|^2]\big)$  and 
$c_2:=4L^2({\bf M}_1^2 T+{\bf M}_2^2)$ 
\begin{equation*}
  \E |X_t-X^{K}_t|^2\le c_1  \int_0^t \|\Delta^K_1(t-s)\|^2+ \|\Delta^K_2(t-s)\|^2 ds+ c_2\int_0^t \bar{G}(t-s)^2 \E \left[| X_s-X^{K}_s|^2\right]ds.
\end{equation*}
Hence, we use the generalized Gronwall\footnote{Note that if $\lambda(\R_+)<\infty$ then $\bar{G}(t-s)^2\le \lambda(\R_+)^2$ and then we can use the classical Gronwall lemma. This argument cannot be applied for the rough kernels.}  Lemma (see e.g. \cite[Theorem 9.8.2]{GriLonStaf}) 
to get
\begin{equation*}
\E |X_t-X^{K}_t|^2\le c_1 \Big(\int_0^t \|\Delta^K_1(t-s)\|^2+ \|\Delta^K_2(t-s)\|^2  ds \Big) \Big( 1+ \int_0^T{\rm E}_{c_2}(s)ds\Big),
\end{equation*}
where ${\rm E}_{c_2}$ is  defined by~\eqref{resolvante}. Since $\Delta_j^K(t-s)=\int_{[K,+\infty)}M_j(\rho) e^{-\rho (t-s)} \lambda(d\rho)$, we have
  $\|\Delta_j^K(t-s)\|\le {\bf M}_j \int_{[K,+\infty)} e^{-\rho (t-s)} \lambda(d\rho)$ and thus 
\begin{align*}
  \int_0^t  \|\Delta_j^K(t-s)\|^2 ds &\le  {\bf M}_j^2\int_0^t \int_{[K,+\infty)}\int_{[K,+\infty)}e^{-(\rho_1+\rho_2)(t-s)}\lambda(d\rho_1) \lambda(d\rho_2) ds\\&={\bf M}_j^2 \int_{[K,+\infty)}\int_{[K,+\infty)}\frac{1- e^{-(\rho_1+\rho_2)t}}{\rho_1+\rho_2}\lambda(d\rho_1) \lambda(d\rho_2)\le {\bf M}_j^2 r(K).
\end{align*}
We therefore get~\eqref{Truncation_estimate} with $C=c_1({\bf M}_1^2+{\bf M}_2^2) \Big( 1+ \int_0^T{\rm E}_{c_2}(s)ds\Big)$. 
\end{proof}
One interest to work with truncation is that the family $G_1^K$ and $G_2^K$ are uniformly bounded in $L^2([0,T])$. However, the proof of Proposition~\ref{prop:tight_r(K)} can easily be extended to obtain the approximation error for  general kernels $\hat{G}_1$ and $\hat{G_2}$ that satisfy
\begin{equation}\label{hypkernels}
  \exists \bar{C} \in \R_+^*, \ \forall j\in \{1,2\},t\in[0,T] \ \|\hat{G}_j(t)\|^2\le \bar{C}(1+\|G_j(t)\|^2),
\end{equation}
so that, by Theorem 3.1~\cite{Zhang}, there exists a unique solution to
 \begin{equation*}\hat{X}_t=x_0+\int_0^t\hat{G}_1(t-s)b(\hat{X}_s)ds + \int_0^t\hat{G}_2(t-s)\sigma (\hat{X}_s)dW_s, t\in[0,T].
 \end{equation*}
This is stated in the next theorem. This theorem completes~\cite[Theorem 3.6]{EleAbi} for the case where coefficients $b$ and $\sigma$ are Lipschitz continuous. 
\begin{theorem}\label{thm_non_asymp}(Non asymptotic estimates) 
Assume~\eqref{SVE}--\eqref{cond_kernel_SVE},\eqref{cond_lambda}--\eqref{cond_Gbar}. Let $\hat{G}_1$ and $\hat{G}_2$ be two kernels satisfying~\eqref{hypkernels}. Then, there exists  a constant $C\in \R^*_+$ (depending on $b$, $\sigma$, $G_1$, $G_2$ and $\bar{C}$) such that
  \begin{equation}\label{non_asymp}
    \E[|\hat{X}_t-X_t|^2]\le C  \Big(\int_0^t \|\hat{G}_1(s)-G_1(s)\|^2+ \|\hat{G}_2(s)-G_2(s)\|^2  ds \Big).
  \end{equation}
\end{theorem}

We now focus on bounding the truncation error $r(K)$. 
\begin{lemma}\label{Lem:trunc:error}Under the assumptions of Proposition~\ref{prop:tight_r(K)}, we have
$r(K)\leq \frac 12 \Big(\int_{[K,+\infty)} \frac{\lambda(d\rho)}{\sqrt{\rho}}\Big)^2.$
 If $\lambda(d\rho)= f(\rho)d\rho$ with $f(\rho)\underset{\rho\to \infty}{=} O(\rho^{-\eta-1/2})$ for some $\eta>0$, we have $r(K)\underset{K\to \infty}=O(K^{-2\eta})$.
\end{lemma}
\begin{proof}
The upper bound is obtained from the standard inequality $\frac{2}{\rho_1+\rho_2}\le \frac 1 {\sqrt{\rho_1\rho_2}}$. For
$\lambda(d\rho)= f(\rho)d\rho$ with $f(\rho)=O(\rho^{-\eta-1/2})$  and $\eta>0$ there is a constant $C>0$ such that 
$f(\rho)<C\rho^{-\eta-1/2}$ for $\rho>1$ and thus 
$$
\int_{[K,+\infty)}\frac{\lambda(d\rho)}{\sqrt{\rho}}\le C\eta K^{-\eta}.
$$
\end{proof}

We now turn to the approximation of the truncated kernel $G^K_j$, $j\in\{1,2\}$ by a kernel $\hat{G}^K_j$.  Let $T>0$.  We define, for $t\in [0,T]$
$$\forall t\in [0,T], \hat{\Delta}^K_j(t)=\hat{G}^K_j(t)-G^K_j(t),$$
and assume the following bound:
\begin{equation}\label{unif_bound_delta}
  \exists \bar{\Delta} :[0,T]\to \R_+ \text{ s.t. } \int_0^T  \bar{\Delta}^2(t)dt<\infty, \ \forall K>0, \forall t\in [0,T], \|\hat{\Delta}^K_j(t)\|\le \bar{\Delta}(t).\tag{H2}
\end{equation}
Note that in our examples, we will use $\bar{\Delta}$ as a constant function, but we keep it general for the presentation of the results. The assumption~\eqref{unif_bound_delta} implies $\int_0^T \|\hat{G}^K_1(t)\|+\|\hat{G}^K_2(t)\|^2dt<\infty$, and we know from Theorem~3.1~\cite{Zhang} that there exists a unique strong solution  $(\hat{X}^K_t,t\in[0,T])$ of the SVE
  $$\hat{X}^K_t=x_0+\int_0^t\hat{G}^K_1(t-s)b(\hat{X}^K_s) ds+ \int_0^t\hat{G}^K_2(t-s)\sigma(\hat{X}^K_s) ds .$$ 
The key property of~\eqref{unif_bound_delta} is that the bound is uniform in~$K$. This enables to get the following result. 
\begin{lemma}\label{lem_uni} (Uniform estimate on $X^K$)
  Let~\eqref{unif_bound_delta} hold. Then, there exists $C\in \R_+^*$ (depending on  $|x_0|$, $T$, $|b(0)|$, $\|\sigma(0)\|$, $L$, ${\bf M}_1$ and ${\bf M}_2$) such that
  $$ \forall K>0,\forall t\in[0,T],\ \E[|\hat{X}_t^K|^2]\le C.$$
\end{lemma}
\begin{proof}
  We have by using Jensen's formula and Itô's isometry, for $t\in[0,T]$,
  $$  \E[|\hat{X}_t^K|^2]\le 3|x_0|^2+3t\int_0^t\|\hat{G}_1^K(t-s)\|^2 \E[|b(\hat{X}^K_s)|^2]ds + 3 \int_0^t \|\hat{G}_2^K(t-s)\|^2 \E[\|\sigma(\hat{X}^K_s)\|^2]ds.$$
  On the one hand, we use that $|b(x)|\le|b(0)|+L|x|$ and $\|\sigma(x)\|\le \|\sigma(0)\| +L|x|$. On the other hand, we get from~\eqref{unif_bound_delta} and~\eqref{unifbound_GK} $\|\hat{G}_j^K(t)\|^2\le 2(\bar{\Delta}^2(t)+\|G_j^K(t)\|^2)\le 2(\bar{\Delta}(t)^2+{\bf M}_j^2\bar{G}(t)^2)$. Since $\int_0^T\bar{\Delta}(t)^2+\bar{G}(t)^2dt <\infty$, this leads to the existence of a constant $C\in \R_+^*$ that depends on $|x_0|$, $T$, $|b(0)|$, $\|\sigma(0)\|$, $L$, ${\bf M}_1$ and ${\bf M}_2$ such that
 $$  \E[|\hat{X}_t^K|^2]\le C+ C\int_0^t \left(\bar{\Delta}(t-s)^2+\bar{G}(t-s)^2\right) \E[|\hat{X}_s^K|^2] ds.$$
  For $c>0$, let  $(\tilde{{\rm E}}_c(t),t\in [0,T])$  be defined as the solution of the equation \begin{equation}\label{resolvante2}\tilde{{\rm E}}_c(t) = \bar{\Delta}^2(t)+\bar{G}^2(t) +\int_0^t c (\bar{\Delta}^2(t-s)+\bar{G}^2(t-s))\tilde{{\rm E}}_c(s)ds.
\end{equation}
  Since $ \bar{\Delta}^2+\bar{G}^2 \in L^1((0,T),\R_+)$, we get that $\tilde{{\rm E}}_c(t)$ is well defined and belongs also to  $L^1((0,T),\R_+)$ by applying the results of Subsection A.3~\cite{EleAbi} and Theorem 2.3.1~\cite{GriLonStaf} to the kernel $\mathbf{1}_{(0,T)}(t)[\bar{\Delta}^2(t)+\bar{G}^2(t)]$. We then get from~\cite[Lemma~A.4]{EleAbi} or~\cite[Lemma 9.8.2]{GriLonStaf}
  $$ \forall t\in [0,T],  \E[|\hat{X}_t^K|^2]\le C\left(1+ \int_0^T \tilde{{\rm E}}_C(t)dt\right),$$
  which gives the claim.
\end{proof}

\begin{prop}\label{prop:approx}
  Let $T>0$. Suppose that for any $K>0$, there are kernels $\hat{G}^K_1,\hat{G}^K_2:[0,T]\to \cM_d(\R)$ such that~\eqref{unif_bound_delta} holds. Then, there is a constant $C\in \R_+^*$ (depending on $|x_0|$, $T$, $|b(0)|$, $\|\sigma(0)\|$, $L$, ${\bf M}_1$ and ${\bf M}_2$) such that
  $$ \forall t \in [0,T], \  \E\left[ |\hat{X}^K_t-X^{K}_t|^2\right]\le  C \left(\int_0^t \left[ \|\hat{\Delta}^K_1(s)\|^2+ \|\hat{\Delta}^K_2(s)\|^2 \right] ds \right). $$
\end{prop}
\begin{proof}
We repeat the same arguments as in the proof of Proposition~\ref{prop:tight_r(K)} and get
 \begin{align*}
&\E\left[ |\hat{X}^K_t-X^{K}_t|^2\right]\le 4t \int_0^t  \|\hat{\Delta}^K_1(t-s)\|^2 \E[|b(\hat{X}^K_s)|]^2ds + 4 \int_0^t  \|\hat{\Delta}^K_2(t-s)\|^2 \E[\|\sigma(\hat{X}^K_s)\|^2]ds\\
   &+4t \int_0^t \|G^{K}_1(t-s)\|^2 \E[ \big| b(\hat{X}^K_s)-b(X^{K}_s)\big|^2]ds + 4 \int_0^t \|G^{K}_2(t-s)\|^2\E[\| \sigma(\hat{X}^K_s)-\sigma(X^{K}_s)\|^2]ds.
 \end{align*}
 From Lemma~\ref{lem_uni}, we get the existence of a constant~$C\in \R_+^*$ such that $$\sup_{K>0}\sup_{t\in [0,T]}\E[|\hat{X}^K_t|^2]\le C.$$
 Then, we set similarly as in the proof of Proposition~\ref{prop:tight_r(K)}
 $c_1:=8(T\vee 1)\big(|b(0)|^2\vee\|\sigma(0)\|^2 +L^2C \big)$,   
$c_2:=4L^2({\bf M}_1^2 T+{\bf M}_2^2)$, and we get 
\begin{equation*}
  \E |\hat{X}^K_t-X^{K}_t|^2\le c_1  \int_0^t \|\hat{\Delta}^K_1(s)\|^2+ \|\hat{\Delta}^K_2(s)\|^2 ds+ c_2\int_0^t \bar{G}(t-s)^2 \E \left[|\hat{X}^K_s-X^{K}_s|^2\right]ds.
\end{equation*}
Hence, we use the generalized Gronwall Lemma (see e.g. \cite[Lemma 9.8.2]{GriLonStaf}) to get
\begin{align*}
\E |\hat{X}^K_t-X^{K}_t|^2&\le c_1 \left( 1+ \int_0^T{\rm E}_{c_2}(s)ds\right) \left(\int_0^t \left[ \|\hat{\Delta}^K_1(s)\|^2+ \|\hat{\Delta}^K_2(s)\|^2 \right] ds \right)  ,
\end{align*}
where ${\rm E}_{c_2}$ is  defined by~\eqref{resolvante}.  
\end{proof}
Combining Propositions~\ref{prop:tight_r(K)} and~\ref{prop:approx}, we obtain easily our main result. 
\begin{theorem}\label{thm_estimee} Let us assume that $\lambda$ satisfies~\eqref{eq_H1} and that~\eqref{unif_bound_delta} holds. Then, there exists a constant $C\in\R_+^*$ such that
  $$ \forall t \in [0,T],\ \E[|X_t-\hat{X}^K_t|^2]\le C\left( r(K)+ \int_0^t \left[ \|\hat{\Delta}^K_1(s)\|^2+ \|\hat{\Delta}^K_2(s)\|^2 \right] ds  \right).$$
\end{theorem}
\noindent The term $r(K)$ and the integral in the right hand side correspond respectively to the truncation and discretization error. When using a Riemann discretization, we get the following general result.
\begin{corollary}\label{cor_approx}Let us assume that $\lambda$ satisfies~\eqref{eq_H1}, and that the functions $M_j:\R_+\to \cM_d(\R)$ are Lipschitz continuous:
  $$ \exists \bar{L}>0,\ \forall j\in\{1,2\},\forall \rho,\rho' \ge 0, |M_j(\rho)-M_j(\rho')| \le \bar{L}|\rho-\rho'|.$$
  Let $n\in \N^*$, $I^K_{i,n}=\left[\frac{i-1}{n}K, \frac  in K \right)$ for $1\le i\le n$ and $\rho^K_{i,n}\in I^K_{i,n}$. Let us define the kernels
    $$j\in \{1,2 \}, \ \hat{G}_j^K(t)=\sum_{i=1}^n \lambda\left(I^K_{i,n}  \right) M_j(\rho^K_{i,n})e^{-\rho^K_{i,n} t},$$
    that correspond to the measure
    \begin{equation}
      \hat{\lambda}(d\rho)=\sum_{i=1}^n \lambda\left(I^K_{i,n}  \right) \delta_{\rho^K_{i,n}}(d\rho).\label{def_hat_lambda}
    \end{equation}
Then, there exists  a constant $C\in\R_+^*$ such that for $n\ge K\lambda([0,K))$, we have 
  $$ \forall t \in [0,T],\ \E[|X_t-\hat{X}^K_t|^2]\le C\left( r(K)+ \frac{K^2}{n^2}\lambda([0,K))^2\right).$$
  \end{corollary} 
\noindent This corollary indicates the theoretical optimal choice for $n$, when $K\to +\infty$. Namely, one has to take $n$ proportional to $\frac{K \lambda([0,K))}{\sqrt{r(K)}}$ in order to equalize both terms, i.e. the error due to the truncation and the one due to the approximation.
\begin{proof}
  We have $\hat{G}_j^K(t)-G_j^K(t)=\sum_{i=1}^n \int_{I^K_{i,n}} \left[ M_j(\rho^K_{i,n})e^{-\rho^K_{i,n} t} -M_j(\rho)e^{-\rho t} \right] \lambda(d \rho)$. From the triangular inequality, we get for $t\in[0,T]$
  \begin{align*}
    \|M_j(\rho^K_{i,n})e^{-\rho^K_{i,n} t} -M_j(\rho)e^{-\rho t}\| &\le \|M_j(\rho^K_{i,n})-M_j(\rho)\| e^{-\rho^K_{i,n} t}+  \|M_j(\rho)\| |e^{-\rho^K_{i,n} t}- e^{-\rho t}| \\
    &\le (\bar{L}+ {\bf M}_jt)| \rho -\rho^K_{i,n}| \le (\bar{L}+ {\bf M}_jT)\frac K n.
  \end{align*}
This yields to $\|\hat{G}_j^K(t)-G_j^K(t)\|\le  (\bar{L}+ {\bf M}_j) \lambda([0,K))\frac K n$, for any $t\in[ 0,T]$. In particular,~\eqref{unif_bound_delta} holds for $n\ge K \lambda([0,K))$. We can thus apply Theorem~\ref{thm_estimee} and get the result. 
\end{proof}
Corollary~\ref{cor_approx} gives a general result on the approximation of SVE by SDE. Obviously, it is possible to derive many variations and refinements of this result by assuming more regularity on the functions $M_j$ or on the measure~$\lambda$. In the next section, we investigate some of these refinements when $\lambda$ is given by~\eqref{frac_mes}.

\section{Euler scheme for Stochastic Volterra Equations}\label{Sec_Euler}

In a recent paper, Richard et al.~\cite{RTY} have proposed and studied the convergence of the following Volterra Euler scheme for the Stochastic Volterra Equation~\eqref{SVE}:
\begin{equation}\label{Euler_RTY}
  X^N_{t_{k+1}}=x_0+\sum_{j=0}^kG_1\left((k+1-j)\frac T N  \right) b(X^N_{t_{j}})\frac{T}{N} +\sum_{j=0}^kG_2\left((k+1-j)\frac T N  \right) \sigma(X^N_{t_{j}})(W_{t_{j+1}}-W_{t_j}),
\end{equation}
for $0\le k\le N-1$, where $T>0$ and $t_k=kT/N$ is the regular time grid.

Note that one of the main drawbacks of the Euler scheme~$X^N_{t_k}$ (with respect to the classical SDE framework) is that it requires  to sum $k$ terms at each time step, so that the overall computational cost is proportional to $N^2$. We will see that for approximating kernels of completely monotone type, we can reduce this to $N\times n$ with $n<<N$ while preserving the same strong rate of convergence, see Theorem~\ref{prop_euler}.

We start by proving a result that plays an analogous role to Theorem~\ref{thm_non_asymp} for this Euler scheme for general kernels $G_1,G_2$ and $\hat{G}_1,\hat{G}_2$.

\begin{theorem}\label{thm_non_asymp_scheme}
  We make the same assumptions as in Theorem~\ref{thm_non_asymp}. Then, there exists  a constant $C\in \R_+^*$ (depending on $b$, $\sigma$, $G_1$, $G_2$ and $\bar{C}$ in~\eqref{hypkernels}) such that
  $$ \max_{0\le k\le N} \E[|\hat{X}^N_{t_k}-X^N_{t_k}|^2]\le C\left( \frac T N \sum_{k=1}^N \|\hat{G}_1(t_k)-G_1(t_k)\|^2 + \|\hat{G}_2(t_k)-G_2(t_k)\|^2\right).$$
\end{theorem}
\begin{remark}
   With respect to Theorem~\ref{thm_non_asymp}, the $L^2$ norm of $\hat{G}_i-G_i$ on $(0,T)$ is replaced by a discrete $L^2$ norm that does not weight on the interval $(0,T/N)$: in the case of exploding kernels at~$0$ like the rough kernel, this discrete norm may be  significantly smaller.
\end{remark}

Combining Theorem~\ref{thm_non_asymp_scheme} with~\cite[Theorem 2.2]{RTY}, we get the following corollary giving the strong error of the Euler scheme $\hat{X}^N$. 
\begin{corollary}\label{cor_euler} Under the assumptions of Theorem~\ref{thm_non_asymp} and if in addition we assume that~\cite[Assumption 2.1]{RTY} holds true with the constant $\alpha>0$ defined therein,  then we have
  \begin{equation*}\max_{0\le k\le N} \E[|\hat{X}^N_{t_k}-X_{t_k}|^2]\le C\left( \left(\frac TN \right)^{2 (\alpha \wedge 1)}+ \frac T N \sum_{k=1}^N \|\hat{G}_1(t_k)-G_1(t_k)\|^2 + \|\hat{G}_2(t_k)-G_2(t_k)\|^2\right).
  \end{equation*}
In dimension one with $G_1=G_2=G_{\lambda_H}$ and $H\in(0,\frac 12)$, we have   $$\max_{0\le k\le N} \E[|\hat{X}^N_{t_k}-X_{t_k}|^2]\le C\left( \left(\frac TN \right)^{2 H}+ \frac T N \sum_{k=1}^N \|\hat{G}(t_k)-G_{\lambda_H}(t_k)\|^2 \right).$$
\end{corollary}

This corollary is a useful tool to analyse the strong rate for any Euler scheme obtained with any approximating kernel. It gives the same rate as the Euler scheme without the approximation kernel, provided that this approximation is accurate. From a practical point of view, for the rough kernel with $H$ small (which corresponds to the financial application), there is no need to be too much accurate for the kernel approximation since the main error comes from the discretization. For example, if $T=1$, $H=0.1$ and $\sqrt{\frac TN \sum_{k=1}^N(\hat{G}(t_k)-G_{\lambda_H}(t_k))^2}\le 0.1$, then to achieve a precision of order $0.1$, one needs to take $N^{-H}=0.1$ , i.e. $N=10^{10}$, which is too much in practice (see e.g. the numerical example~\eqref{val_num_approx}). This results allows us to build alternative Euler approximation schemes having the same accuracy but with a smaller time complexity, see Theorem~\ref{prop_euler}.

\begin{proof}[Proof of Theorem~\ref{thm_non_asymp_scheme}]
  We have by using $(a+b)^2\le 2a^2+2b^2$ and the Itô isometry
  \begin{align*}
    \E[|\hat{X}^N_{t_{k+1}}-X^N_{t_{k+1}}|^2]\le &2 \E \left[\left(\frac{T}{N}\right)^2 \left|\sum_{j=0}^k\left(\hat{G}_1(t_{k+1-j}) b(\hat{X}^N_{t_j})-G_1(t_{k+1-j})b(X^N_{t_j})\right)\right|^2\right] \\
     & + 2 \sum_{j=0}^k \E \left[\left\|\hat{G}_2(t_{k+1-j})  \sigma(\hat{X}^N_{t_j})-G_2(t_{k+1-j})\sigma(X^N_{t_j})\right\|^2 \right] \frac{T}N.
  \end{align*}
  By using the Cauchy-Schwarz inequality on the first sum, it can be then analysed as the second sum, and we assume without loss of generality from now on that $b=0$.
  We get
  \begin{align*} \E[|\hat{X}^N_{t_{k+1}}-X^N_{t_{k+1}}|^2]\le & 4 \sum_{j=0}^k \E \left[\left\|G_2(t_{k+1-j}) [ \sigma(\hat{X}^N_{t_j})- \sigma(X^N_{t_j})] \right\|^2 \right] \frac{T}N  \\&+  4 \sum_{j=0}^k \E \left[\left\|[\hat{G}_2(t_{k+1-j}) -G_2(t_{k+1-j})]\sigma(\hat{X}^N_{t_j})\right\|^2 \right] \frac{T}N
  \end{align*}
  First, we get by using~\eqref{hypkernels} and \cite[Proposition 4.1]{RTY} that $\sup_{0\le j\le N}\E[|X^N_{t_j}|^2]\le C$. From the Lipschitz property~\ref{cd:lip}, we have  $\|\sigma(x)\|\le \|\sigma(0)\|+Lx$ and thus
  \begin{align*} \E[|\hat{X}^N_{t_{k+1}}-X^N_{t_{k+1}}|^2]\le & 4 L^2 \sum_{j=0}^k \E \left[\|G_2(t_{k+1-j})\|^2 |\hat{X}^N_{t_j}- X^N_{t_j}|^2 \right] \frac{T}N  \\&+  C \sum_{j=0}^k \left\|\hat{G}_2(t_{k+1-j}) -G_2(t_{k+1-j})\right\|^2  \frac{T}N\\
    \le & 4 L^2 {\bf M}_2 \frac{T}N\sum_{j=0}^k \bar{G}^2(t_{k+1-j}) \E \left[|\hat{X}^N_{t_j}- X^N_{t_j}|^2 \right] +  C \epsilon_N ,  
  \end{align*}
  with $\epsilon_N=\frac{T}N\sum_{j=1}^{N} \left\|\hat{G}_2(t_{j}) -G_2(t_{j})\right\|^2$ and where the constant $C>0$ may change from one line to another.  
  Since $\int_0^T\bar{G}^2(t)dt<\infty$ by Assumption~\ref{cond_Gbar}, we conclude the proof by applying Lemma~\ref{lem_discrete_gronwall}, which can be seen as a discrete version of the Generalized Gronwall Lemma~\cite[Theorem 9.8.2]{GriLonStaf}.
\end{proof}

\begin{lemma}[A generalized discrete Gronwall Lemma]\label{lem_discrete_gronwall}
  Let $\lambda$ be a measure on $\R_+$ such that $K(t)=\int_{\R_+}e^{-\rho t} \lambda(d\rho) <\infty$ for any $t>0$ and such that $\int_0^T K(t)^pdt<\infty$ for some $T,p>0$. For a given $N\in \N^*$, we consider a finite sequence $(u^N_k)_{0 \le k \le N} $ of nonnegative real numbers such that $u^N_0=0$ and
  $$u^N_{k+1}\le C_1 + C_2 \frac{T}N \sum_{j=0}^k K((k+1-j)\frac T N)^pu^N_j, \ k\in\{0,\dots,N-1\},$$
  for some $C_1,C_2>0$. Then, there exists a constant $M \in \R_+^*$ depending only on $(C_2,p,T)$ and on the  kernel $K$ such that:
  $$ \max_{0\le k\le N} u^N_k  \le C_1 M. $$
\end{lemma}
\begin{proof}
  By the dominated convergence theorem, we may find $a>0$ large enough depending only on the kernel $K$ and on $(C_2,p,T)$ such that $C_2\int_0^T e^{-at}K(t)^pdt <1/2$. We note $t_k=kT/N$ and have
  \begin{align*}
    e^{-at_{k+1}}u_{k+1}^N &\le C_1e^{-at_{k+1}}+ C_2 \frac{T}N \sum_{j=0}^k e^{-at_{k+1-j}}K(t_{k+1-j})^p e^{-at_j}u^N_j \\
      & \le C_1+ C_2 \left( \max_{0\le j\le k} e^{-a t_j}u^N_j \right) \frac{T}N \sum_{j=1}^{k+1} e^{-at_{j}}K(t_{j})^p.
  \end{align*}
  Since the function $t \mapsto e^{-at} K(t)^p$ is continuous nonincreasing on $(0,T]$, we obtain that   $\frac{T}N \sum_{j=1}^{k+1} e^{-at_{j}} K(t_{j})^p \le \frac{T}N \sum_{j=1}^{N} e^{-at_{j}} K(t_{j})^p\le  \int_0^T e^{-at} K(t)^pdt<1/(2C_2)$. Therefore,  we obtain that 
$e^{-at_{k+1}}u^N_{k+1}\le C_1+ \frac 12 \max_{0\le j\le k} e^{-at_j}u^N_{j}$. Let us denote $m_k=\max_{0\le j\le k} e^{-at_j}u^N_j $. We thus have $m_{k+1}\le \max(m_k,\frac{m_k}2+C_1)$. Since $m_0=0$, and the function $x\mapsto \max(x,x/2+C_1)$ is nondecreasing with fixed point $2C_1>0$ we get by induction that $\forall k\in \{0,\dots,N\}, m_k\le 2C_1$. We conclude by remarking that $\max_{0\le k\le N} u^N_k\le e^{aT}m_N\le 2C_1e^{aT} $.
\end{proof}


We now turn to the second main result of this section: for a completely monotone kernel, the Euler scheme on the SDE~\eqref{SDE_approx_2} (or~\eqref{SDE_approx_1}) is essentially the same as the Euler scheme proposed by Richard et al.~\cite{RTY}  on the corresponding SVE. Let us be more precise and consider two kernels $G_1(t)=\sum_{i=1}^n\alpha_iM_1(\rho_i) e^{-\rho_i t}$ and $G_2(t)=\sum_{i=1}^n\alpha_iM_2(\rho_i) e^{-\rho_i t}$ with $\alpha_1,\dots,\alpha_n\ge 0$ and $0\le \rho_1<\rho_2<\dots<\rho_n$. At a first glance, we could directly write the Euler scheme for the SDE~\eqref{SDE_approx_2}, but we have noticed in practice that for large values of $\rho_i$ (typically when $\rho_i T/N >>1$), the approximation of the part of the drift term which is proportional to~$\rho_i$ may not be accurate. This typically happens when approximating a completely monotone kernel. To overcome this problem, we write the multifactor Euler scheme associated to $(\bar{X}^i_t,\bar{Y}^i_t)=e^{\rho_i t}(X^{\rho_i}_t-x_0^i,Y^{\rho_i}_t-y_0^i)$. From~\eqref{SDE_approx_2}, we easily get $X_t=x_0+\sum_{j=1}^n\alpha_jM_1(\rho_j)e^{-\rho_j t}\bar{X}^j_t + \sum_{j=1}^n\alpha_jM_2(\rho_j)e^{-\rho_j t}\bar{Y}^j_t$ with
\begin{align*}
  d\bar{X}^i_t&=e^{\rho_i t}b\left(x_0+\sum_{j=1}^n\alpha_jM_1(\rho_j)e^{-\rho_j t}\bar{X}^j_t + \sum_{j=1}^n\alpha_jM_2(\rho_j)e^{-\rho_j t}\bar{Y}^j_t\right)dt,\\
  d\bar{Y}^i_t&=e^{\rho_i t}\sigma\left(x_0+\sum_{j=1}^n\alpha_jM_1(\rho_j)e^{-\rho_j t}\bar{X}^j_t + \sum_{j=1}^n\alpha_jM_2(\rho_j)e^{-\rho_j t}\bar{Y}^j_t\right)dW_t.
\end{align*}
This leads to the following multifactor Euler scheme, for $0\le k\le N-1$,
\begin{align*}
  \bar{X}^{i,N}_{t_{k+1}}&=\bar{X}^{i,N}_{t_k}+ e^{\rho_i t_k}b\left(x_0+\sum_{j=1}^n\alpha_jM_1(\rho_j)e^{-\rho_j t_k}\bar{X}^{j,N}_{t_k} + \sum_{j=1}^n\alpha_jM_2(\rho_j)e^{-\rho_j t_k}\bar{Y}^{j,N}_{t_k}\right)\frac{T}{N},\\
  \bar{Y}^{i,N}_{t_{k+1}}&=\bar{Y}^{i,N}_{t_{k}}+e^{\rho_i t_k}\sigma\left(x_0+\sum_{j=1}^n\alpha_jM_1(\rho_j)e^{-\rho_j t_k}\bar{X}^{j,N}_{t_k} + \sum_{j=1}^n\alpha_jM_2(\rho_j)e^{-\rho_j t_k}\bar{Y}^{j,N}_{t_k}\right)(W_{t_{k+1}}-W_{t_k}),
\end{align*}
with $\bar{X}^{i,N}_{t_0}=\bar{Y}^{i,N}_{t_0}=0$. Equivalently, we may set $\hat{X}^{i,N}_{t_k}=e^{-\rho_i t_k}\bar{X}^{i,N}_{t_k}$ and $\hat{X}^{i,N}_{t_k}=e^{-\rho_i t_k}\bar{X}^{i,N}_{t_k}$. Then, we get
\begin{align}
  \hat{X}^{i,N}_{t_{k+1}}&=e^{-\rho_i \frac TN}\left(\hat{X}^{i,N}_{t_{k}}+ b\left(\hat{X}^N_{t_k}\right)\frac{T}{N}\right), \ \hat{Y}^{i,N}_{t_{k+1}}=e^{-\rho_i \frac TN}\left(\hat{Y}^{i,N}_{t_{k}}+ \sigma \left(\hat{X}^N_{t_k}\right)(W_{t_{k+1}}-W_{t_k})\right),\label{def_Euler_Scheme_SDE_2}\\
  \hat{X}^N_{t_k}&=x_0+\sum_{i=1}^n\alpha_iM_1(\rho_i)\hat{X}^{i,N}_{t_k} + \sum_{i=1}^n\alpha_i M_2(\rho_i)\hat{Y}^{i,N}_{t_k}. \notag
\end{align}
In the case where $G_1(t)=G_2(t)=\sum_{i=1}^n\alpha_iM(\rho_i) e^{-\rho_i t}$, we can similarly define the multifactor Euler scheme associated to~\eqref{SDE_approx_1} by  $\hat{X}^{i,N}_{t_{0}}=0$ and
\begin{align}
  \hat{X}^{i,N}_{t_{k+1}}&=e^{-\rho_i \frac TN} \left(\hat{X}^{i,N}_{t_{k}} + b\left(\hat{X}^N_{t_k}\right)\frac{T}{N}+\sigma \left(\hat{X}^N_{t_k} \right)(W_{t_{k+1}}-W_{t_k})\right) , \label{def_Euler_Scheme_SDE}\\
 \hat{X}^N_{t_k} &=x_0+\sum_{i=1}^n\alpha_iM(\rho_i)\hat{X}^{i,N}_{t_k}.\notag
\end{align}
\begin{theorem}\label{prop_euler}
  Let us assume that $G_1(t)=\sum_{i=1}^n\alpha_iM_1(\rho_i) e^{-\rho_i t}$ and $G_2(t)=\sum_{i=1}^n\alpha_iM_2(\rho_i) e^{-\rho_i t}$. Then, the Euler schemes~\eqref{Euler_RTY} and~\eqref{def_Euler_Scheme_SDE_2} coincides, i.e. $\forall k\in \{0, \dots,N\}$, $X^{N}_{t_k}=\hat{X}^N_{t_k}$. Similarly, when $G_1(t)=G_2(t)=\sum_{i=1}^n\alpha_iM(\rho_i) e^{-\rho_i t}$, the Euler schemes~\eqref{Euler_RTY} and~\eqref{def_Euler_Scheme_SDE} coincides.  
\end{theorem}

\begin{proof}
  We prove this result by induction on~$k$. We have $X^N_{t_0}=\hat{X}^N_{t_0}=x_0$, and assume that for $0\le k< N$,  $X^N_{t_j}=\hat{X}^N_{t_j}$ for all $j\in \{0,\dots,k\}$. Then, we have by~\eqref{def_Euler_Scheme_SDE_2}
  \begin{align*}
    \hat{X}^N_{t_{k+1}}&=x_0+\sum_{i=1}^n\alpha_iM_1(\rho_i)\hat{X}^{i,N}_{t_{k+1}} + \sum_{i=1}^n\alpha_iM_2(\rho_i)\hat{Y}^{i,N}_{t_{k+1}}
  \end{align*}
We have $\hat{X}^{i,N}_{t_{k+1}}= e^{-\rho_i T/N}(\hat{X}^{i,N}_{t_k}+ b(\hat{X}^N_{t_k})T/N)$, and we get by induction on $k$ and using $\hat{X}^{i,N}_{t_{0}}=0$,
$$\hat{X}^{i,N}_{t_{k+1}}=\sum_{j=1}^k e^{-\rho_i (k+1-j) \frac TN} b(\hat{X}^N_{t_j}) \frac TN.$$
We similarly have $\hat{Y}^{i,N}_{t_{k+1}}=\sum_{j=1}^k e^{-\rho_i (k+1-j) \frac TN} \sigma(\hat{X}^N_{t_j}) (W_{t_{j+1}}-W_{t_j})$ and then
\begin{align*}
  \hat{X}^N_{t_{k+1}}=&x_0+\sum_{j=1}^k \sum_{i=1}^n\alpha_iM_1(\rho_i)e^{-\rho_i (k+1-j) \frac TN} b(\hat{X}^N_{t_j}) \frac TN \\
  &+ \sum_{j=1}^k \sum_{i=1}^n\alpha_i M_2(\rho_i) e^{-\rho_i (k+1-j) \frac TN} \sigma(\hat{X}^N_{t_j}) (W_{t_{j+1}}-W_{t_j})\\
  =&x_0+ \sum_{j=1}^k G_1\left((k+1-j) \frac TN\right)  b(\hat{X}^N_{t_j})\frac TN +\sum_{j=1}^k  G_2\left((k+1-j) \frac TN\right)  \sigma(\hat{X}^N_{t_j})(W_{t_{j+1}}-W_{t_j}),
\end{align*}
which proves the first claim by using the induction hypothesis and~\eqref{Euler_RTY}. We get the second claim with the same arguments. 
\end{proof}

To implement the Euler scheme, the formulas~\eqref{def_Euler_Scheme_SDE_2} and~\eqref{def_Euler_Scheme_SDE} only require a computational cost proportional to $n \times N$. For $n<<N$, this is much faster than computing the sums in~\eqref{Euler_RTY}. Therefore, to approximate the SVE~\eqref{SVE} with kernels of the form~\eqref{kernel_form}, two strategies are possible: we can either use the Euler scheme~\eqref{Euler_RTY} or approximate the kernels and use~\eqref{def_Euler_Scheme_SDE_2}. A thorough comparison between is beyond the scope of this paper, but we will show in the numerical Section~\ref{subsec_rHeston} the relevance of the second approach for the rough Heston model.

Moreover, the multifactor scheme~\eqref{def_Euler_Scheme_SDE_2} (resp.~\eqref{def_Euler_Scheme_SDE}) provides a universal multidimensional approximation of~\eqref{SVE} (resp.~\eqref{SVE} with $G_1=G_2=G$) that can be used for any $\alpha_i$ and $\rho_i$, independently on the method used to fit $\hat{G}_1(t)=\sum_{j=1}^n \alpha_jM_1(\rho_j)e^{-\rho_j t}$ and $\hat{G}_2(t)=\sum_{j=1}^n \alpha_jM_2(\rho_j)e^{-\rho_j t}$ to the given kernels $G_1$ and $G_2$ (resp. $\hat{G}(t)=\sum_{j=1}^n \alpha_jM(\rho_j)e^{-\rho_j t}$ to the given kernel $G$).  

We now discuss the possibility of reducing the value of~$n$. In practice, when approximating kernels, it may happen that we find very large values of the  $\rho$'s exponential coefficients. This is typically the case for kernels that are unbounded around~$0$, such as the rough kernel. In this case, we observe that for large values of $\rho_i$ it is useless to simulate $\hat{X}^{i,N}$ and $\hat{Y}^{i,N}$ since they remain close to zero as $e^{-\rho_i \frac TN}<<1$. We therefore introduce  for $\beta>0$
\begin{equation}\label{def_ntilde}
  \tn=\inf \{ k \in \{1,\dots,n\} :  ({\bf M}_1\vee {\bf M}_2) \sum_{i=k+1}^n \alpha_i   e^{-\rho_i \frac T N} \le \left(\frac TN\right)^\beta  \text{ for } i\ge k+1 \}.
\end{equation}
Then, we define the following schemes for $k\in \{0,\dots,N-1\}$,
\begin{align}
  \tilde{X}^{i,N}_{t_{k+1}}&=e^{-\rho_i \frac TN}\left(\tilde{X}^{i,N}_{t_{k}}+ b\left(\tilde{X}^N_{t_k}\right)\frac{T}{N}\right), \ \tilde{Y}^{i,N}_{t_{k+1}}=e^{-\rho_i \frac TN}\left(\tilde{Y}^{i,N}_{t_{k}}+ \sigma \left(\tilde{X}^N_{t_k}\right)(W_{t_{k+1}}-W_{t_k})\right), 1\le i\le \tn,\label{def_Euler_Scheme_SDE_2_tilde}\\
  \tilde{X}^N_{t_k}&=x_0+\sum_{i=1}^{\tn}\alpha_iM_1(\rho_i)\tilde{X}^{i,N}_{t_k} + \sum_{i=1}^{\tn}\alpha_i M_2(\rho_i)\tilde{Y}^{i,N}_{t_k}, \notag
\end{align}
with $\tilde{X}^{i,N}_{t_0}=\tilde{Y}^{i,N}_{t_0}=0$, and in the case where $G_1=G_2$: 
\begin{align}
  \tilde{X}^{i,N}_{t_{k+1}}&=e^{-\rho_i \frac TN} \left(\tilde{X}^{i,N}_{t_{k}} + b\left(\tilde{X}^N_{t_k}\right)\frac{T}{N}+\sigma \left(\tilde{X}^N_{t_k} \right)(W_{t_{k+1}}-W_{t_k})\right), 1\le i\le \tn, \label{def_Euler_Scheme_SDE_tilde}\\
 \tilde{X}^N_{t_k} &=x_0+\sum_{i=1}^{\tn}\alpha_iM(\rho_i)\tilde{X}^{i,N}_{t_k}.\notag
\end{align}
The advantage of this new procedure is that it reduces the computational complexity to $\tn\times N$, which is a clear gain compared to~\eqref{Euler_RTY}, \eqref{def_Euler_Scheme_SDE_2}, and~\eqref{def_Euler_Scheme_SDE}. For sake of simplicity we analyse the associated error only in the case $G_1=G_2$. 
\begin{corollary}\label{lem_ntilde}  Let $\hat{X}^N$ and $\tilde{X}^N$ be respectively defined by~\eqref{def_Euler_Scheme_SDE} and~\eqref{def_Euler_Scheme_SDE_tilde}. Under the assumptions of Theorem~\ref{thm_non_asymp}, there exists a constant $C\in \R_+$  such that
  $$ \forall k\in \{0,\dots,N\}, \ \E[|\hat{X}^{N}_{t_{k}}-\tilde{X}^N_{t_{k}}|]\le C\left(\frac T N \right)^{2 \beta}.$$
\end{corollary}
\begin{proof}
  Let $\hat{G}(t)=\sum_{i=1}^n \alpha_i e^{-\rho_i t}$ and  $\tilde{G}(t)=\sum_{i=1}^{\tilde{n}} \alpha_i e^{-\rho_i t}$. By definition of $\tilde{n}$~\eqref{def_ntilde}, we get for $t\ge T/N$,
$$\hat{G}(t)-\tilde{G}(t)= \sum_{i=\tilde{n}+1}^n \alpha_i e^{-\rho_i t}\le \sum_{i=\tilde{n}+1}^n \alpha_i e^{-\rho_i T/N}\le \left(\frac{T}{N}\right)^\beta. $$
  We then apply Theorem~\ref{thm_non_asymp}.
\end{proof}


\section{More approximation results for the rough kernels}\label{Sec_Rough}

Let us start by applying the result of Corollary~\ref{cor_approx} to the measure $\lambda$ defined in Equation~\eqref{frac_mes}. We have $\lambda([0,K))=c_H(1/2-H) K^{\frac12 -H}$ and $r(K)=O(K^{-2H})$ by Lemma~\ref{Lem:trunc:error}, which gives
   $$ \forall t \in [0,T],\ \E[|X_t-\hat{X}^K_t|^2]\le C\left(K^{-2H}+ \frac{K^{3 -2H}}{n^2}\right).$$
  By taking $n=K^{\frac32}$ or equivalently $K=n^{\frac 23}$, we get \begin{equation}\label{rate23}\E[|X_t-\hat{X}^K_t|^2]\underset{n\to \infty}=O(n^{-2H \times\frac{2}{3}}).
  \end{equation}
  Let us recall that $n$ is the number of points weighted by the approximating measure~$\hat{\lambda}$. By Proposition~\ref{prop:SVE-SDE}, $n$ scales as the dimension of the SDE that approximates the SVE and therefore as the computation time needed to simulate the SDE. The goal of this section is to improve this rate, by assuming more regularity on the functions $M_j$.
  
  To get a better approximation, we assume more regularity on the functions $M_1$ and $M_2$. To approximate $G^K_j(t)=\int_0^K e^{-\rho t} M_j(\rho) c_H \rho^{-H-1/2}d\rho$, we use the same type of approximation on $[0,K^{\beta}]$ with $0<\beta<1$ and then use the Simpson's rule on $[K^{\beta},K]$, with $K>1$.
\begin{prop}\label{prop_Simpson}
Suppose that $\lambda$ is given by~\eqref{frac_mes}. Let us assume that the functions  $M_1$ and $M_2$ are $\mathcal{C}^4$ with bounded derivatives. Let $\beta \in (0,1)$ and $\hat{G}^K_j(t)=\int_{\R_+} e^{-\rho t}M_j(\rho) \hat{\lambda}^S(d\rho)$ with
\begin{align*}
  \hat{\lambda}^{S}(d\rho)=&\sum_{i=1}^n \lambda(I^{K^\beta}_{i,n}) \delta_{\rho^{K^\beta}_{i,n}}\\&
  +\frac{c_H(K-K^\beta)}{6n}\sum_{i=1}^n \Big[(\rho_{i,n,0}^K)^{-H-\frac12}\delta_{\rho_{i,n,0}^K} + 
4(\rho_{i,n,1}^K)^{-H-\frac12}\delta_{\rho_{i,n,1}^K}+(\rho_{i,n,2}^K)^{-H-\frac12}\delta_{\rho_{i,n,2}^K}\Big],
\end{align*}
where $I^{K^\beta}_{i,n}=[\frac{i-1}{n}K^\beta, \frac{i}{n}K^\beta)$, $\rho^{K^\beta}_{i,n} \in I^{K^\beta}_{i,n}$,  $\rho_{i,n,0}^K=K^\beta+\frac{(i-1)(K-K^\beta)}{n}$,  $\rho_{i,n,1}^K=K^\beta+\frac{(2i-1)(K-K^\beta)}{2n}$ and  $\rho_{i,n,2}^K=K^\beta+\frac{i(K-K^\beta)}{n}$.
  With $\beta=\frac{10-6H}{13-6H}$ and $K\sim n^{\frac{13-6H}{15-6H}}$, there exists a constant $C\in \R_+^*$ such that 
  $$\forall t \in [0,T],\ \E[|\hat{X}^K_t-X_t|^2]\le C n^{-2H\times \frac{13-6H}{15-6H}} .$$
\end{prop}
We clearly have $\frac 56 \le \frac{13-6H}{15-6H}$ for $H\in(0,1/2)$ and notice that $\hat{\lambda}^S$ weights $3n+1=O(n)$ different points. Thus, the approximation given by Proposition~\ref{prop_Simpson} is asymptotically better than the one given by Corollary~\ref{cor_approx}. 

\begin{proof}
  We aim at applying Theorem~\ref{thm_estimee}. We have
  \begin{align*} &\hat{G}^K_j(t)-G^K_j(t)=\sum_{i=1}^n \int_{I^{K^\beta}_{i,n}} \left[ M_j(\rho^{K^\beta}_{i,n})e^{-\rho^{K^\beta}_{i,n} t} -M_j(\rho)e^{-\rho t} \right] \lambda(d \rho) -\int_{K^\beta}^K M_j(\rho)e^{-\rho t} c_H\rho^{-H-1/2}d \rho  \\
    &+\frac{c_H(K-K^\beta)}{6n}\sum_{i=1}^n \Big[(\rho_{i,n,0}^K)^{-H-\frac12}M_j(\rho_{i,n,0}^K) + 
4(\rho_{i,n,1}^K)^{-H-\frac12}M_j(\rho_{i,n,1}^K)+(\rho_{i,n,2}^K)^{-H-\frac12}M_j(\rho_{i,n,2}^K)\Big].
  \end{align*}
  The norm of the first sum can be upper bounded by $O\left(\lambda([0,K^\beta))\frac {K^\beta} n\right)=O\left(\frac {K^{\beta(3/2-H)}} n\right)$,  as in the proof Corollary~\ref{cor_approx}. For the other terms, we work componentwise and may assume w.l.o.g. that $M_j$ is real valued. Let $\psi_t(\rho)=c_HM_j(\rho)\rho^{-H-1/2}e^{-\rho t}$.  The well known convergence result on the Simpson's rule (see e.g. \cite{IsKe}, p.\ 339) allows to upper bound the norm of the other terms by
    $$ \frac{\sup_{\rho\in[K^\beta, K]} \psi_{t}^{(4)}(\rho) }{90n^4}\Big(\frac{K-K^\beta}{2}\Big)^5.$$
 We get that $\sup_{t\in[0,T]}\sup_{\rho\in[K^\beta, K]} \psi_{t}^{(4)}(\rho)={{O}}(K^{-\beta(H+1/2)})$ by using that the derivatives of $M_j$ are bounded and $0\le e^{-\rho t}\le 1$. This leads to
  \begin{align}
\forall t\in[0,T],\    \| \hat{G}^K_j(t)-G^K_j(t)\| \le C\left( \frac {K^{\beta(3/2-H)}} n + \frac{K^{5-\beta(H+1/2)}}{n^4}\right).
  \end{align}
Note that~\eqref{unif_bound_delta} is then satisfied for $n\ge \max(K^{\beta(3/2-H)},K^{[5-\beta(H+1/2)]/4})$. Then, by Theorem~\ref{thm_estimee} and Lemma~\ref{Lem:trunc:error}, we then get
  $$\forall t\in[0,T],\ \sqrt{\E[|\hat{X}^K_t-X_t|^2]}\le C \left(K^{-H}+ \frac {K^{\beta(3/2-H)}} n + \frac{K^{5-\beta(H+1/2)}}{n^4}\right).$$
By taking $\beta=\frac{10-6H}{13-6H}$ and $K\sim n^{\frac{13-6H}{15-6H}}$, we equalize the three terms and get the claim. 
\end{proof}

We can now go further and use higher order numerical integration algorithm such as the Newton-Cotes method, which for any even number $J\in \N$ and any smooth function $f: [a,b]\rightarrow \R$ gives (see e.g.  \cite[Theorem~1, p. 310]{IsKe})
 $$ \int_a^b f(x)dx=(b-a)\sum_{j=0}^J c^J_j f(a+j\frac{b-a}{J}) + \tilde c^J (b-a)^{J+3}f^{(J+2)}(\xi), \text{ with } \xi\in(a,b), $$  
 where the coefficients $(c^J_j)_{0\le j\le J }$ and $ \tilde c^J $ are known explicitly.  We recover the Simpson's rule  by taking $J=2$. Hence, one can use the Newton-Cotes method on the interval $[K^\beta, K]$. This leads to a new measure 
 \begin{align}\label{lambda_NC}
 {\hat\lambda}^{NC}(d\rho)=&\sum_{i=1}^n \lambda(I^{K^\beta}_{i,n}) \delta_{\rho^{K^\beta}_{i,n}} +\frac{c_H(K-K^\beta)}{n}\sum_{i=1}
^n \sum_{j=0}^J  c^J_j (\rho^{K,J}_{i,n,j})^{-H-\frac12}\delta_{\rho^{K,J}_{i,n,j}} 
\end{align}
 with $\rho^{K,J}_{i,n,j}=K^\beta+\frac{K-K^\beta}{n}( i-1 +\frac jJ)$.
 \begin{prop}\label{prop_NC}
   Suppose that $\lambda$ is given by~\eqref{frac_mes}. Let us assume that the functions  $M_1$ and $M_2$ are $\mathcal{C}^\infty$ with bounded derivatives. Let $J \in \N$ and $\hat{G}^K_j(t)=\int_{\R_+} e^{-\rho t}M_j(\rho) \hat{\lambda}^{NC}(d\rho)$ with $\hat{\lambda}^{NC}$ defined by~\eqref{lambda_NC}.
  With $\beta=\frac{2(J+3)-2(J+1)H}{3J+7-2(J+1)H}$ and $K\sim n^{\frac{3J+7-2(J+1)H}{3J+9-2(J+1)H}}$, there exists a constant $C\in \R_+^*$ such that 
  $$\forall t \in [0,T], \ \E[|\hat{X}^K_t-X_t|^2]\le C n^{-2H\times \frac{3J+7-2(J+1)H}{3J+9-2(J+1)H}} .$$
For any $\varepsilon\in(0,1)$, there exists $J$ such that $\sup_{t\in[0,T]} \E[|\hat{X}^K_t-X_t|^2]=O(n^{-2H\times(1-\varepsilon)})$.
 \end{prop}
 We note that we get back Proposition~\ref{prop_Simpson} in the case $J=2$. 
\begin{proof}
  We follow the same arguments as in the proof of Proposition~\ref{prop_Simpson}. The terms corresponding to the Newton-Cotes method can be upper bounded by
$|\tilde c^J| \frac{\sup_{\rho\in[K^\beta, K]} \psi_{t}^{(J+2)}(\rho) }{n^{J+2}}\Big({K-K^\beta}\Big)^{J+3}$, that is uniformly $O\left(\frac{K^{J+3-\beta(H+1/2)}}{n^{J+2}}\right)$ in $t\in[0,T]$. We get 
$$\forall t\in[0,T], \ 
 \|\hat{G}^K_j(t)-G^K_j(t)\|\le C \left(\frac{K^{\beta(3/2-H)}}n + \frac{K^{J+3-\beta(H+1/2)}}{n^{J+2}} \right),$$
and then by Corollary~\ref{cor_approx} and Lemma~\ref{Lem:trunc:error}, we obtain
  \begin{equation}\label{intermed_prop_NC}\forall t\in[0,T],\ \sqrt{\E[|\hat{X}^K_t-X_t|^2]}\le C \left(K^{-H}+ \frac {K^{\beta(3/2-H)}} n + \frac{K^{J+3-\beta(H+1/2)}}{n^{J+2}} \right).
  \end{equation}
With $\beta=\frac{2(J+3)-2(J+1)H}{3J+7-2(J+1)H}$ and $K\sim n^{\frac{3J+7-2(J+1)H}{3J+9-2(J+1)H}}$, the three terms are of the same order and we get the first claim. We get the second claim noticing that $ \frac{3J+7-2(J+1)H}{3J+9-2(J+1)H} \underset{J\to +\infty} \to 1$.  
\end{proof}

In dimension $d=1$ with $M_1=M_2\equiv1$, it is possible to take a particular value for $\rho_{i,n}^{K}$ in $I^K_{i,n}$ that improves the rate of convergence. This is stated in the next proposition.

\begin{prop}\label{prop_1d}
  Let us assume that $d=1$ and $M_1=M_2\equiv 1$. Let us define
  $$\rho_{i,n}^K=\frac{\int_{I^K_{i,n}}\rho \lambda(d\rho)}{\lambda(I^K_{i,n})}.$$
  \begin{enumerate}
  \item Let $\hat{\lambda}(d\rho)$ be defined by~\eqref{def_hat_lambda} with these particular values for $\rho_{i,n}^K$. Then, the approximation $\hat{G}^K_j(t)=\int e^{-\rho t} \hat{\lambda}(d\rho)$ with $K\sim n^{\frac 45}$ leads to 
    $$\exists C>0, \forall t \in [0,T], \E[|\hat{X}^K_t-X_t|^2]\le C n^{-2H \times \frac 45}.$$
  \item Let $\hat{\lambda}^{NC}(d\rho)$ be defined by~\eqref{lambda_NC} with these particular values for $\rho_{i,n}^{K^\beta}$.  Then, the approximation $\hat{G}^K_j(t)=\int e^{-\rho t} \hat{\lambda}^{NC}(d\rho)$ with
 $\beta =\frac{4J+12-2HJ}{ 5J+12-2HJ }$ and  $K=n^{\frac 4{5\beta+2H(1-\beta)}}$
leads to 
$$\exists C>0, \forall t \in [0,T], \E[|\hat{X}^K_t-X_t|^2]\le C n^{-2H \times \frac{5J+12-2HJ}{5J+15-2HJ}}.$$
In particular for Simpson's rule ($\hat{\lambda}^{S}$), we get $\sup_{t\in [0,T]} \E[|\hat{X}^K_t-X_t|^2]=O( n^{-2H \times \frac{22-4H}{25-4H}})$. 
  \end{enumerate}
\end{prop}
\noindent It is worth noticing that for the one-dimensional setting, the rate of convergence with factor $\frac{4}{5}$ obtained in the first statement is the same as the one obtained by Abi Jaber and El Euch~\cite{EleAbi} on the kernels $G_j$ and their discrete approximating kernels $\hat{G}_j^K$. Here, we  get in addition a strong estimation error on the processes with the same rate. Note that the factor  $\frac{4}{5}$  improves the factor $\frac 23$ obtained in~\eqref{rate23}, when the values of $\rho^K_{i,n}$ are only assumed to be in $I^K_{i,n}$. Similarly, we notice that
$$ \frac{3J+7-2(J+1)H}{3J+9-2(J+1)H} <\frac{5J+12-2HJ}{5J+15-2HJ}<1,$$
which shows that the convergence rate is improved with respect to Proposition~\ref{prop_NC} but the factor still remains under~1. Very recently, for the one-dimensional setting, Harms~\cite{Harms} has obtained an arbitrary rate of convergence $O(n^{-r})$ by using quadrature rules with $m>3r/2H$ points on a geometric discretization grid in~$\rho$ with $n$ intervals. However, the constant $C$ such that $\sup_{t\in [0,T]} \sqrt{\E[|\hat{X}^K_t-X_t|^2]}\le C n^{-r}$ may be quite large. To be more precise, the constant $C$ given by~\cite[Lemma 2]{Harms} scales as $(1/\eta)^m$ (since the constant $C_3$ defined there scales as $1/\eta$), where $\eta$ may be close to zero and $m$ quite large.  This is confirmed by the numerical expermient~\cite[Figure 3]{Harms} where for example, an error of $10^{-2}$ is obtained for $H=0.1$ with $m=20$ and about $25$ intervals, which makes $25\times20=500$ exponential factors, while in the present paper we obtain very good approximations with less than~$80$ exponential factors, see~Table~\ref{table_approx}. Besides, in practice the asymptotic rate of convergence is not the only issue. Since one approximates the SVE by an SDE with a $O(n)$ times higher dimension, one is rather interested to use a not to high value of~$n$. We will  discuss of this in the next numerical section, see Subsection~\ref{subsec_Improvement}. 
\begin{proof}[Proof of Proposition~\ref{prop_1d}]
  For the first assertion,  we remark that
  $$|G^K_j(t)-\hat{G}^K_j(t)|=\left| \sum_{i=1}^n \int_{I^K_{i,n}}(e^{-\rho t}-e^{-\rho_{i,n}^K t})\lambda(d\rho)\right|\le\sum_{i=1}^n\left|\int_{I^K_{i,n}}(e^{-\rho t}-e^{-\rho_{i,n}^K t})\lambda(d\rho)\right|.$$
  From a Taylor expansion, we get
  $$ e^{-\rho t}-e^{-\rho_{i,n}^K t}= -t(\rho-\rho_{i,n}^K)e^{-\rho_{i,n}^K t}+\int_{\rho_{i,n}^K}^\rho t^2 e^{-x t}(\rho-x)dx.$$
  When integrating with respect to $\lambda$ over $I_{i,n}^K$ , the first term vanishes and we get
  \begin{align*}
    \left| \int_{I_{i,n}^K}(e^{-\rho t}-e^{-\rho_{i,n}^K t})\lambda(d\rho) \right| &= \left|\int_{I_{i,n}^K}\int_{\rho_{i,n}^K}^\rho t^2 e^{-x t}(\rho-x)dx\lambda(d\rho) \right|
    \\ &
    \le t^2\int_{I_{i,n}^K}\int_{\rho_{i,n}^K}^\rho  |\rho-x|dx\lambda(d\rho)=\frac{t^2}{2}\int_{I_{i,n}^K} (\rho-\rho_{i,n}^K)^2\lambda(d\rho)  \le \frac{t^2 K^2}{2n^2}\lambda\left(I_{i,n}^K\right)
  \end{align*}
  since $\rho_{i,n}^K \in I^K_i$. Summing over $i$, we get 
 \begin{equation}\label{ineq:ker}
|G^K_j(t)-\hat{G}^K_j(t)| \le  \frac{t^2 K^2}{2n^2}\lambda\left([0,K]\right).
 \end{equation}
 Thus, \eqref{unif_bound_delta} holds for $n\ge K\sqrt{\lambda\left([0,K]\right)}$. By Theorem~\ref{thm_estimee} and Lemma~\ref{Lem:trunc:error}, we get the existence of $C\in \R_+^*$ such that
 $$\forall t \in [0,T], \E[|X^K_t-X_t|^2]\le C\left(K^{-2H}+ \frac{ K^{5-2H}}{n^4}\right).$$
 This leads to the claim with $K\sim n^{\frac 45}$. 

 For the proof of the second point, we use the result of the first point and repeat the arguments of the Proof of Proposition~\ref{prop_NC}. We thus get
 $$\forall t\in[0,T],\ \sqrt{\E[|\hat{X}^K_t-X_t|^2]}\le C \left(K^{-H}+ \frac {K^{\beta(5/2-H)}}{ n^2} + \frac{K^{J+3-\beta(H+1/2)}}{n^{J+2}} \right). $$
 instead of~\eqref{intermed_prop_NC}. Taking $\beta =\frac{4J+12-2HJ}{ 5J+12-2HJ }$ and  $K=n^{\frac 4{5\beta+2H(1-\beta)}}$ makes the three terms of the same order and leads to the result. The case $J=2$ corresponds to Simpson's rule. 
\end{proof}

\section{Numerical experiments}\label{Sec_Num}

\subsection{Validation of the theoretical results}

The aim of this section is to illustrate the different convergence rates on a very simple example for the rough kernel~\eqref{frac_mes}. Namely, we take $b(x)=0$, $\sigma(x)=1$, which means that
 $$X_t=X_0+\frac 1{\Gamma(H+1/2)} \int_0^t(t-s)^{H-\frac12}dW_s .$$
 For this process, we have implemented the four following approximations.
 \begin{enumerate}
 \item $\hat{X}^{1,n}_t$ the approximation given by~Corollary~\ref{cor_approx} with $K=n^{\frac23}$, $\rho^K_{i,n}=\frac{i-1/2}n K$.
   From~\eqref{rate23}, the theoretical rate of convergence is $\E[|\hat{X}^{1,n}_t-X_t|]=O(n^{-H\times \frac23})$. 
 \item $\hat{X}^{2,n}_t$ the approximation given by Proposition~\ref{prop_1d} with $\hat{\lambda}$ and  $K=n^{\frac45}$.
   The theoretical rate of convergence is $\E[|\hat{X}^{2,n}_t-X_t|]=O(n^{-H\times \frac45})$.
 \item $\hat{X}^{3,n}_t$ the approximation given by Proposition~\ref{prop_Simpson} with $K=n^{\frac{13-6H}{15-6H}}$ and $\rho^{ K^{\frac{10-6H}{13-6H}}}_{i,n}=\frac{i-1/2}n K^{\frac{10-6H}{13-6H}}$. The theoretical rate of convergence is $\E[|\hat{X}^{3,n}_t-X_t|]=O(n^{-H\times \frac{13-6H}{15-6H} })$.
 \item $\hat{X}^{4,n}_t$ the approximation given Proposition~\ref{prop_1d} with $\hat{\lambda}^S$, $K=n^{\frac{22-4H}{25-4H}}$. The theoretical rate of convergence is $\E[|\hat{X}^{4,n}_t-X_t|]=O(n^{-H\times \frac{22-4H}{25-4H} })$.   
 \end{enumerate}
 Note that for $0\le \rho_1<\dots<\rho_n$ it is possible to simulate exactly the Gaussian vector
 $$\left(\int_0^t \exp(-\rho_1(t-s))dW_s,\dots,\int_0^t \exp(-\rho_n(t-s))dW_s,\frac 1{\Gamma(H+1/2)} \int_0^t(t-s)^{H-\frac12}dW_s \right).$$
 It is centered with covariance matrix $\Sigma$ such that
 \begin{align}
   \Sigma_{i,j}&=\frac{1-\exp(-(\rho_i+\rho_j)t)}{\rho_i+\rho_j} \text{ for }1\le i,j\le n, \notag\\
   \Sigma_{n+1,n+1}&=\frac 1{2H \Gamma(H+1/2)^2}t^{2H}, \label{def_Sigma}\\
 \Sigma_{i,n+1}&=\rho_i^{-H-1/2}\int_0^{\rho_it}\frac 1{\Gamma(H+1/2)}s^{H-1/2}e^{-s}ds.\notag
 \end{align}
 The last quantity involves the incomplete gamma function that can be calculated efficiently. For each $j\in \{1,\dots,4\}$, we have calculated, using the following basic lemma, the quantity
 \begin{equation*}
   \zeta^{j,n}_t:=\E[|\hat{X}^{j,n}_t-X_t|^2].
 \end{equation*}
 We reported the obtained results in Tables \ref{table1}--\ref{table4}.
 \begin{lemma} \label{lem_calcul_L2}
   Let $0\le \rho_1<\dots<\rho_n$, $\alpha_1,\dots,\alpha_n\in \R$. Then, 
 $$\sum_{i=1}^n\alpha_i \int_0^t \exp(-\rho_n(t-s))dW_s  -\frac 1{\Gamma(H+1/2)} \int_0^t(t-s)^{H-\frac12}dW_s $$
 is a centered Gaussian random variable with variance $$\int_0^t\left( \frac{(t-s)^{H-\frac12}}{\Gamma(H+1/2)}-\sum_{i=1}^n \alpha_i \exp(-\rho_n(t-s))\right)^2ds= v^\top \Sigma v,$$ where $\Sigma$ is defined by~\eqref{def_Sigma} and $v\in \R^{n+1}$ is defined by $v_i=\alpha_i$ for $1\le i\le n$ and $v_{n+1}=-1$. 
 \end{lemma}

 We have calculated $\zeta^{j,n}_t$ with $n=50$ and $n=100$ for $j\in\{1,2\}$ and $n=16$ and $n=32$ for $j\in\{3,4\}$. Since the measure $\hat{\lambda}^S$ weights $3n+1$ points, this corresponds to approximate with SDEs of dimension $49$ and $97$, making the comparison with the case  $j\in\{1,2\}$ relevant. We have also calculated
 $$\hat{\gamma}^{j,n}_t= \frac 1 {2H\log(2)}\log(\zeta^{j,n}_t/\zeta^{j,2n}_t), $$
 as a numerical estimation of the speed of convergence factor. Indeed, if we had $\E[|\hat{X}^{j,n}_t-X_t|^2]\sim_{n\to \infty} c n^{-2H\times \gamma}$ for some constants $c,\gamma>0$,  then $\hat{\gamma}^{j,n}_t$ would estimate the factor~$\gamma$. In our work, we have obtained  $$\E[|\hat{X}^{j,n}_t-X_t|^2] =_{n\to \infty} O( n^{-2H\times \gamma}),$$ and we have reported this theoretical value of $\gamma$ in the tables below. 
 \begin{table}[h!]
   \begin{tabular}{|r||l|l|l|}
     \hline
     $H$& 0.45 & 0.25 & 0.05 \\
     \hline
     $\zeta^{1,n}_1$ &0.00443 & 0.0547 & 2.1404 \\
     $\zeta^{1,2n}_1$ & 0.00279 & 0.0432 &  2.0436\\
     $\hat{\gamma}^{1,n}_1$ & 0.7433 & 0.6848 & 0.6678  \\
     Theoretical factor & 2/3 & 2/3 & 2/3 \\
     \hline
   \end{tabular}
   \caption{Convergence results for the first approximation, with $n=50$}\label{table1}
 \end{table}
 \begin{table}[h!]
   \begin{tabular}{|r||l|l|l|}
     \hline
     $H$& 0.45 & 0.25 & 0.05 \\
     \hline
     $\zeta^{2,n}_1$ &  0.00024 & 0.0413 & 2.0313  \\
     $\zeta^{2,2n}_1$ &  0.00015 & 0.0313 &1.9218  \\
     $\hat{\gamma}^{2,n}_1$ & 0.80020  & 0.80016  & 0.80003 \\
     Theoretical factor & 0.8 & 0.8 & 0.8 \\
     \hline
   \end{tabular}
   \caption{Convergence results for the second approximation, with $n=50$}\label{table2}
 \end{table}
 \begin{table}[h!]
   \begin{tabular}{|r||l|l|l|}
     \hline
     $H$& 0.45 & 0.25 & 0.05 \\
     \hline
     $\zeta^{3,n}_1$ & 0.00627  & 0.0628  & 2.1869 \\
     $\zeta^{3,2n}_1$ & 0.00357  & 0.0462 &  2.0594 \\
     $\hat{\gamma}^{3,n}_1$ & 0.9064  & 0.8838  &  0.8669 \\
     Theoretical factor & 0.8374 &  0.8519 &  0.8639 \\
     \hline
   \end{tabular}
   \caption{Convergence results for the third approximation, with $n=16$}\label{table3}
  \end{table}
  \begin{table}[h!]
   \begin{tabular}{|r||l|l|l|}
     \hline
     $H$& 0.45 & 0.25 & 0.05 \\
     \hline
     $\zeta^{4,n}_1$ & 0.00046 & 0.0588  & 2.177   \\
     $\zeta^{4,2n}_1$ & 0.00027  & 0.0434  & 2.048  \\
     $\hat{\gamma}^{4,n}_1$ & 0.8713  & 0.8754 & 0.8792 \\
     Theoretical factor & 0.8707 &  0.875 &  0.8790 \\
     \hline
   \end{tabular}
   \caption{Convergence results for the fourth approximation, with $n=16$}\label{table4}
  \end{table}
  
  \vspace{3cm}
  From these numerical results, we observe the following facts:
  \begin{itemize}
  \item For each method, the quality of the approximation downgrades as $H$ gets closer to~$0$. For $H=0.05$, even if we observe empirical rates of convergence that are in line with our theoretical results, the approximation error is around~$2$ for all methods, which is clearly too large for practical use. The next subsection presents significant improvements for this issue.

  \item We notice that the numerical estimation of the speed of convergence factor is always above the theoretical value of $\gamma$. These values coincide quite well for the one-dimensional methods (2nd and 4th methods) and for the case $H=0.05$ for all methods. For the approximations~1 and~3 and the values $H=0.45$ and $0.25$, the theoretical value  of the speed of convergence factor seems to be slightly pessimistic.
  \item The improvement due to the particular choice of $\rho^K_{i,n}$ in dimension~1 is significant. The values of $\zeta^{2,n}_1$ and $\zeta^{2,2n}_1$ (resp. $\zeta^{4,n}_1$ and $\zeta^{4,2n}_1$) are significantly smaller than the one of $\zeta^{1,n}_1$ and $\zeta^{1,2n}_1$ (resp.   $\zeta^{3,n}_1$ and $\zeta^{3,2n}_1$).
  \item The asymptotic acceleration of convergence obtained by Simpson's rule (i.e. by using approximation~3 (resp.~4) instead of~1 (resp.~2)) is not yet observed for these values of~$n$. The approximation~1 (resp. 2)  with $n=50$ gives a slightly better result than approximation~3 (resp. 4) with $n=16$.   
    
  \end{itemize}
  
  \subsection{Improvement of the approximations for the rough kernel : a systematic approach}\label{subsec_Improvement}
  
  In practice, the method provided by truncating and discretizing the integral $\int_0^{+\infty}e^{-\rho t} M(\rho)\lambda(d\rho)$ is partly satisfactory. Its advantage is that it is systematic, and it may lead to good rates of convergence when $\lambda(d\rho)$ has a thin tail and under smoothness assumption. For the rough kernel, $\lambda(d\rho)=c_H\rho^{-H-1/2}d\rho$ is not smooth close to the origin and  has fat tails, which makes the truncation error large. Thus, the convergences  that we obtain in Section~\ref{Sec_Rough} are quite slow, especially when $H$ is close to zero. Here, we present a systematic way to correct this by truncating at a higher level. 

 The principle is the following. All the methods that we have presented consists in truncating the integral $\int_{\R_+} e^{-\rho t } \lambda_H(d\rho)$ at $K=n^{\gamma H}$ for some $\gamma>0$ and then to use a discretization scheme on $[0,K]$. Here, in addition, we take $A>1$ and approximate the integral on $[K,A^nK)$ by using the same discretization rule on each interval $[A^{i-1}K,A^{i}K)$ for $i=1,\dots,n$. Since the size of these intervals does not go to zero, we do not expect to improve the asymptotic rate of convergence: the goal is rather to reduce the truncation error.

      For simplicity, we present this idea only on the approximation~$\hat{\lambda}$ given by Proposition~\ref{prop_1d}.  Namely, let $K>0$ and we define for $i\in \{1,\dots,2n\}$,
\begin{equation}
    I^{K, A}_{i,n}=\left[\frac{i-1}nK,\frac i n K \right) \text { for }i\le n, \  I^{K, A}_{i,n}=\left[KA^{i-n-1}, KA^{i-n} \right) \text { for }n+1 \le i\le 2n.
\end{equation}
We then consider for $i\le 2n$, $\rho^{K,A}_{i,n}=\frac{\int_{I^{K, A}_{i,n}}\rho \lambda_H(d\rho) }{\int_{I^{K, A}_{i,n}}\lambda_H(d\rho)}$, which can be calculated exactly since
$$\frac{\int_{[a,b]}\rho \lambda_H(d\rho)}{\int_{[a,b]} \lambda_H(d\rho)}= \frac{1/2-H}{3/2-H}\times \frac{b^{3/2-H}-a^{3/2-H}}{b^{1/2-H}-a^{1/2-H}} \text{ for } 0\le a<b.$$
Last, we define the corresponding approximating measure $\hat{\lambda}^A$ by
\begin{equation}\label{def_lambdaA}
  \hat{\lambda}^A(d\rho)=\sum_{i=1}^{2n}\lambda_H( I^{K, A}_{i,n}) \delta_{\rho^{K,A}_{i,n}}(d\rho),
\end{equation}
and $\hat{G}^{K,A}(t)= \int_{\R_+} e^{-\rho t}\hat{\lambda}^A(d\rho)$. We have the simple but interesting result. 
\begin{prop}\label{prop_1d_2}
  Let $\hat{\lambda}^A(d\rho)$ be defined by~\eqref{def_lambdaA}, $\hat{\lambda}(d\rho)=\sum_{i=1}^{n}\lambda_H( I^{K, A}_{i,n}) \delta_{\rho^{K,A}_{i,n}}(d\rho)$ be the measure introduced in Proposition~\ref{prop_1d} and $\hat{G}^K(t)=\int_{\R_+} e^{-\rho t}\hat{\lambda}(d\rho)$. Then, we have
  $$ \hat{G}^K(t)\le \hat{G}^{K,A}(t) \le G_{\lambda_H}(t).$$
  If $X$ (resp. $\hat{X}^{K,A}$) denotes the solution of $X_t=x_0+\int_0^tG_{\lambda_H}(t-s)b(\hat{X}_s)ds +\int_0^tG_{\lambda_H}(t-s)\sigma(X_s)dW_s$ (resp. $\hat{X}^{K,A}_t=x_0+\int_0^t\hat{G}^{K,A}(t-s)b(\hat{X}^{K,A}_s)ds +\int_0^t\hat{G}^{K,A}(t-s)\sigma(\hat{X}^{K,A}_s)dW_s$), we have $\E[|\hat{X}^{K,A}_t-X_t|^2]=O(n^{-2H\times \frac45})$ if $K\sim_{n\to \infty} c n^{4/5}$ for some $c\in \R_+^*$. 
\end{prop}
\begin{proof}
  The first inequality is obvious. The second one is a consequence of Jensen inequality that gives $\int_I e^{-\rho t} \lambda_H(d\rho) \ge \lambda_H(I) e^{-\frac{\int_{I}\rho \lambda_H(d\rho) }{\int_{I}\lambda_H(d\rho)} t}$ on any interval $I$ since $\rho\to e^{-\rho t}$ is a convex function. We then get
   $0\le G_{\lambda_H}(t) -\hat{G}^{K,A}(t)\le  G_{\lambda_H}(t) -\hat{G}^{K}(t)$ and thus $\int_0^T (G_{\lambda_H}(t) -\hat{G}^{K,A}(t))^2dt \le \int_0^T (G_{\lambda_H}(t) -\hat{G}^{K}(t))^2dt$ for any $T>0$. This gives by~\eqref{ineq:ker}, Theorem~\ref{thm_estimee} and Lemma~\ref{Lem:trunc:error} the rate of convergence.
\end{proof}

Note that Proposition~\ref{prop_1d_2} gives the same asymptotic rate of convergence than Proposition~\ref{prop_1d}. This is confirmed on our numerical experiments: we have indicated in Table~\ref{table-improved} the $L^2$-errors obtained with $K=n^{4/5}$ and $A=3$ and the estimated rate of convergence~$\hat{\gamma}$ that is close to the theoretical one of $4/5$. However, comparing with Table~\ref{table2} (approximation by $\hat{G}^K$), we see that the error is significantly reduced: for $n=50$ and $H=0.05$, we get a squared error of $0.0112$ instead $2.03$. Thus, if
the rate of convergence is not improved with respect to the approximation given by~$\hat{\lambda}$, the approximation given by~$\hat{\lambda}^A$ significantly reduces the approximation error. 
This suggests that the kernel with the constant~$A$ improves the multiplicative constant in the rate of convergence.
 \begin{table}[H]
   \begin{tabular}{|r||l|l|l|}
     \hline
     $H$& 0.45 & 0.25 & 0.05 \\
     \hline
     $\zeta^{50}_1$ & $1.631 \times 10^{-6}$ & $8.305\times 10^{-5}$ &0.01120   \\
     $\zeta^{200}_1$ & $5.866 \times 10^{-7}$ & $4.567\times 10^{-5}$ &0.002547     \\
     $\zeta^{400}_1$ & $3.520 \times 10^{-7}$ & $3.412\times 10^{-5}$ &0.002408 \\
     $\hat{\gamma}:=\frac 1 {2H\log(2)}\log(\zeta^{200}_1/\zeta^{400}_1)$ & 0.819 & 0.841 & 0.806    \\
     Theoretical factor & 0.8 &  0.8 &  0.8 \\
     \hline
   \end{tabular}
   \caption{Convergence results for $\zeta^n_t=\E[|\hat{X}^{K,A}_t-X_t|^2]$, with $A=3$ and $t=1$. }\label{table-improved}
 \end{table}
 
\begin{figure}
    \centering
  \subfigure[][$H=0.45$, $n=5$]{ \includegraphics[width=0.47\linewidth]{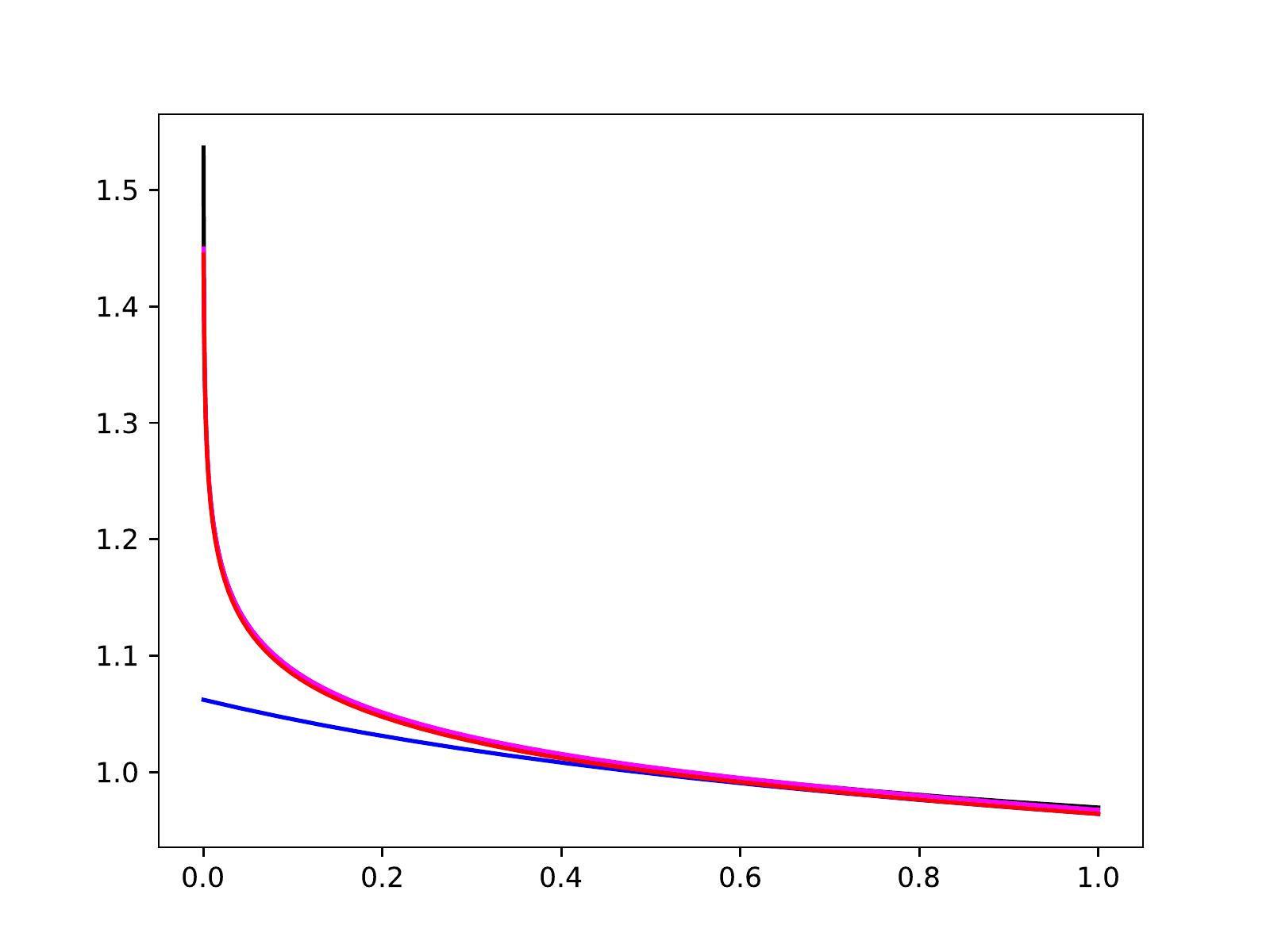}  \label{n5H045}}
  \subfigure[][$H=0.45$, $n=10$]{\includegraphics[width=0.47\linewidth]{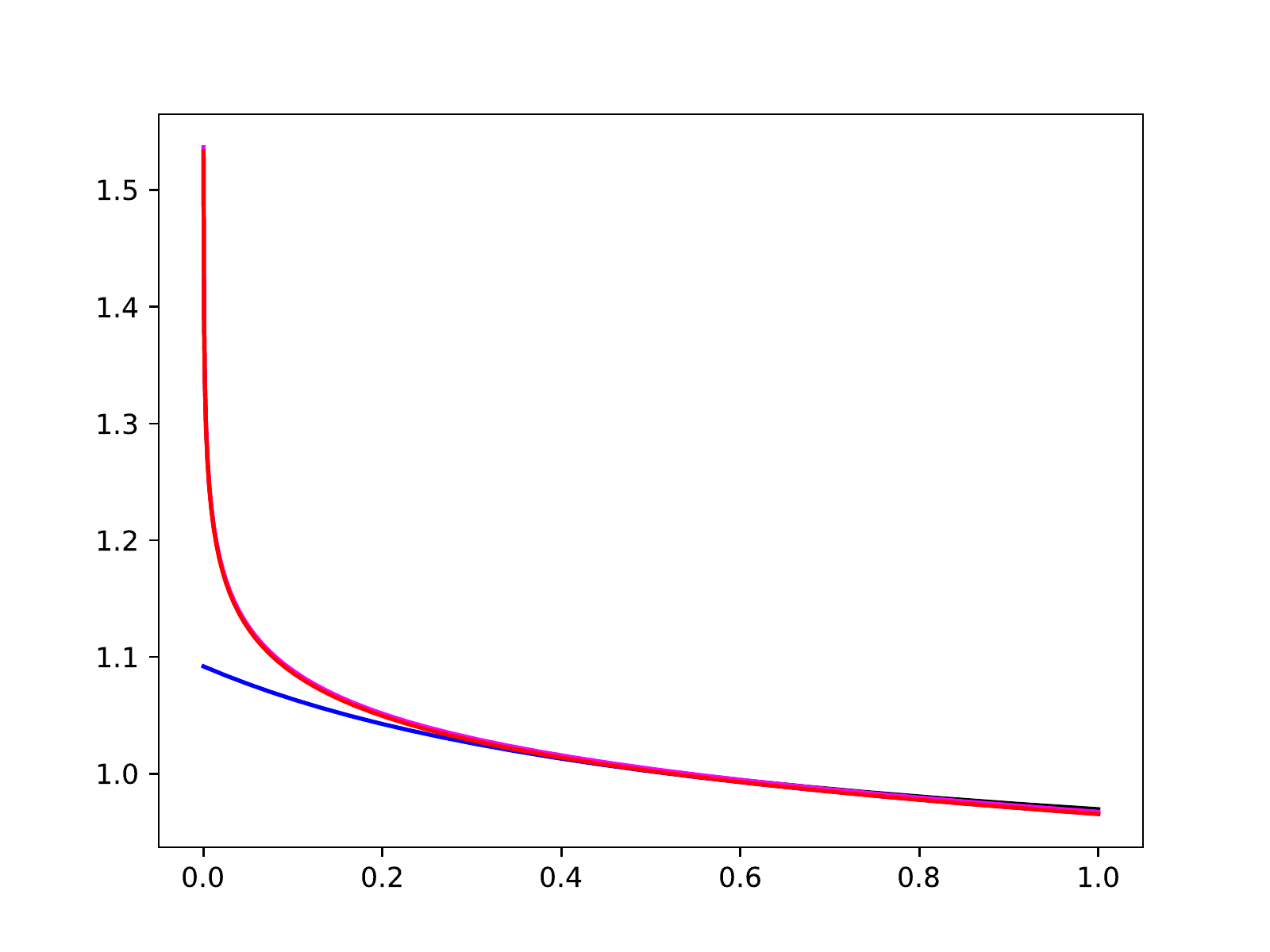} }
  
  \subfigure[][$H=0.25$, $n=5$]{  \includegraphics[width=0.48\linewidth]{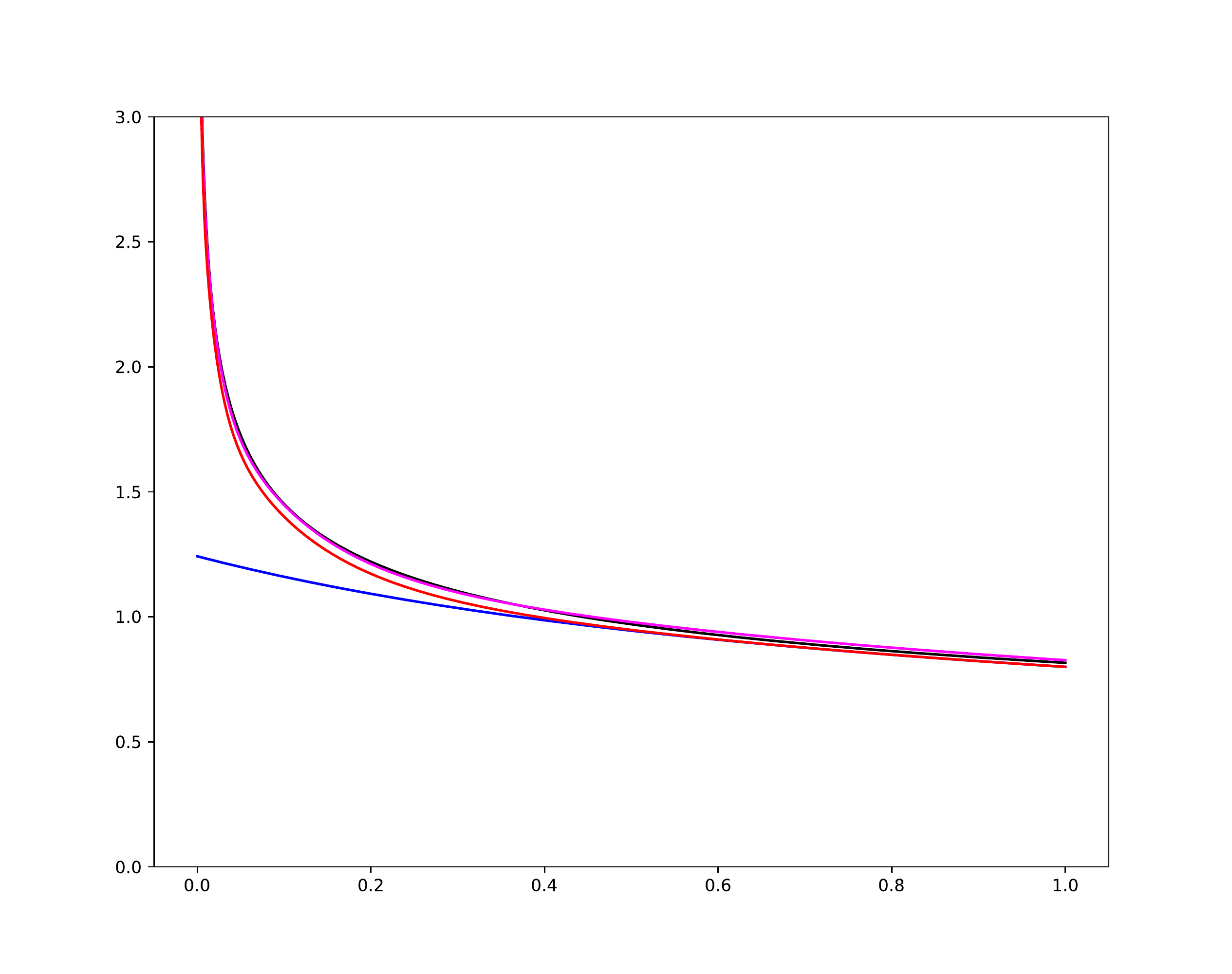}}
  \subfigure[][$H=0.25$, $n=10$]{ \includegraphics[width=0.46\linewidth]{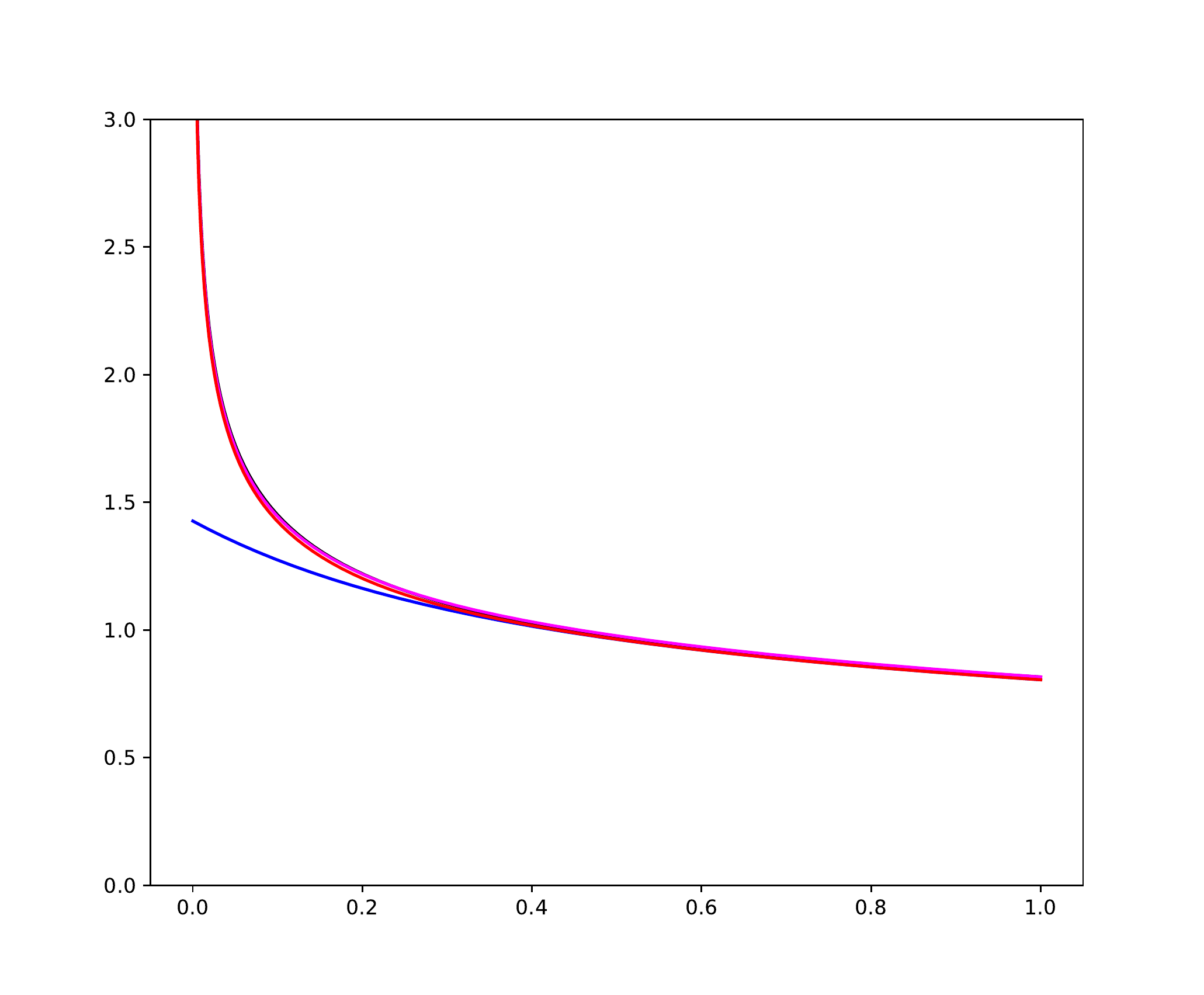}}
  \subfigure[][$H=0.05$, $n=10$]{  \includegraphics[width=0.47\linewidth]{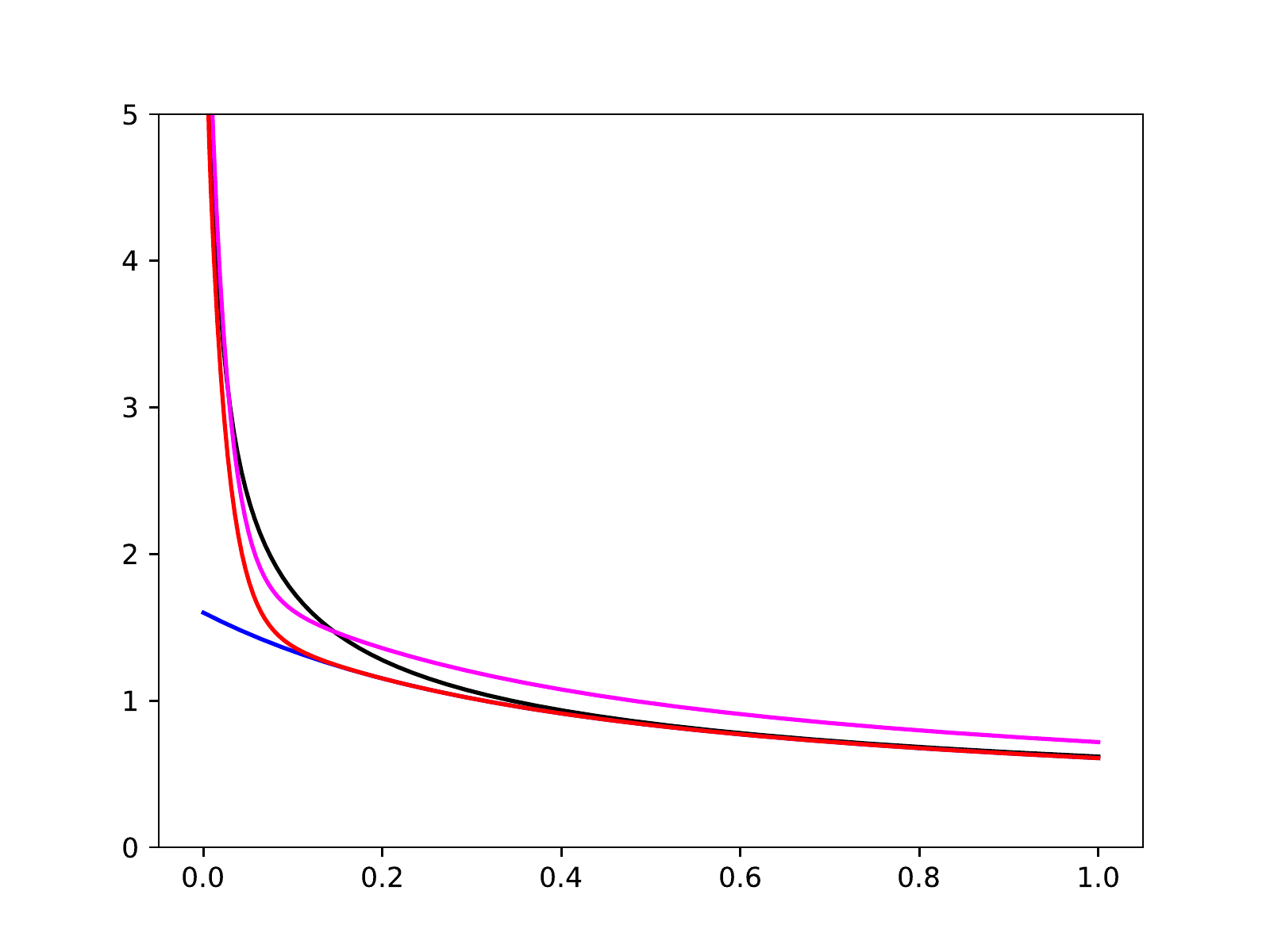}}
  \subfigure[][$H=0.05$, $n=40$]{ \includegraphics[width=0.47\linewidth]{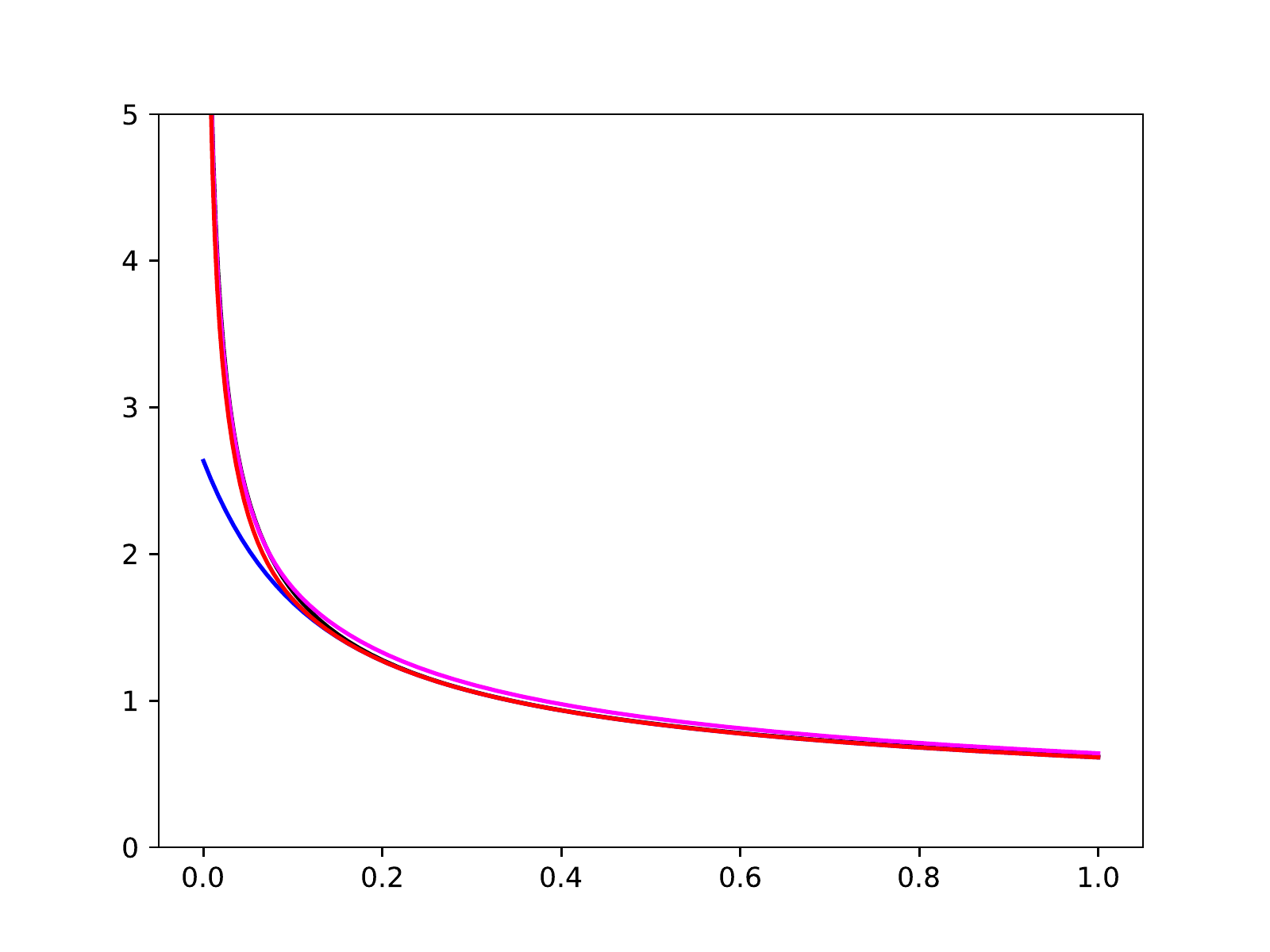}}
  \caption{Plots of $G_{\lambda_H}(t)=\frac{t^{H-1/2}}{\Gamma(H+1/2)}$ (black), $\hat{G}^{n^{4/5}}(t)$ (blue), $\hat{G}^{n^{4/5},A^* }(t)$ (red) and $\hat{G}^{sys}_{2n}(t)$ (magenta) for different values of $H$ and $n$.}\label{figure_G}
  \end{figure}

 Now, we discuss the choice of~$A$. By Theorem~\ref{thm_non_asymp}, $\int_0^T(G_{\lambda_H}(t)-\hat{G}^{K,A}(t))^2dt$ is a natural criterion to assess the quality of the approximation. Besides, we know by Lemma~\ref{lem_calcul_L2} that this quantity can be calculated easily. Thus, it is natural to find $A^*$ that minimizes $\int_0^T(G_{\lambda_H}(t)-\hat{G}^{K,A}(t))^2dt$. This can be done in practice by using a one-dimensional optimization routine.

 Last, once $A^*$ has been calculated, we still notice that we have $\hat{G}^{K,A^*}(t)\le G_{\lambda_H}(t)$ by Proposition~\ref{prop_1d_2}. Therefore, there exists $\xi^*\ge 1$ that minimizes $\int_0^T(G_{\lambda_H}(t)-\xi\hat{G}^{K,A^*}(t))^2dt$, namely
 $$ \xi^* = \frac{ \int_0^T G_{\lambda_H}(t)\hat{G}^{K,A^*}(t)dt}{ \int_0^T (\hat{G}^{K,A^*}(t))^2dt},  $$
 that can similarly as in Lemma~\ref{lem_calcul_L2} be calculated exactly by the mean of the Gamma incomplete function. Let us note that with this last adjustment, the approximation $\xi^*G^{K,A^*}$ is still completely monotone, which may be an interesting property to preserve.

 Figure~\ref{figure_G} illustrates for different values of $H$ the different approximations of the rough kernel. It shows the  interest of the progressive steps  of our approximations from $\hat{G}^{n^{4/5}}(t)$ to $\hat{G}^{n^{4/5},A^* }(t)$ and then to  $\xi^*\hat{G}^{n^{4/5},A^*}(t)$.  Here, and from now on, we set for $n\in 2\N^*$
 \begin{equation}\label{def_sys}
   \hat{G}^{sys}_n(t):= \xi^*\hat{G}^{(n/2)^{4/5},A^*}(t),
 \end{equation}
the approximation obtained with the systematic approach that uses a combination of $n$ exponential functions.

 We first observe that the approximation  $\hat{G}^{n^{4/5}}(t)$ provided by Proposition~\ref{prop_1d} is not accurate close to time zero, due to the truncation. For $H=0.45$ (resp. $H=0.25$), the approximation provided by $\hat{G}^{n^{4/5},A^* }(t)$ and $\hat{G}^{sys}_{2n}(t)$ are quite perfect for $n=5$ (resp. $n=10$). For $H=0.05$ and $n=10$, one better observes the role of the parameter~$\xi^*$ that shifts upward the approximation so that it crosses $G_{\lambda_H}$ at some optimal point to minimize the $L^2$ error. For $n=40$ the approximation of the rough kernel is quite perfect. We have indicated in Table~\ref{table_approx} the corresponding $L^2$ errors between $\hat{G}^{sys}_{n}$ and the rough kernel for different values of $H$ and~$n$.
  
    \begin{table}[H]
   \begin{tabular}{|r||l|c|}
     \hline
     $H$& $n$  & $\sqrt{\int_0^1( G_{\lambda_H}(t)-\hat{G}^{sys}_{n}(t))^2dt}$ \\
     \hline
      0.45 & 10   & 0.00209   \\  
     0.45 & 20   &  0.00107  \\ 
     \hline
      0.25 & 20   & 0.0134   \\ 
      0.25 & 40   &  0.0049  \\
      \hline
       0.05 & 40  & 0.189   \\
     0.05 & 80   &  0.084  \\ 
     \hline
   \end{tabular}
   \caption{Values of the $L^2$ error between the rough kernel and its approximation by the systematic approach.}\label{table_approx}
    \end{table}

\subsection{Application to the Rough Bergomi model}
In this subsection, we give a practical application and consider the pricing of European call options with the Rough Bergomi model. This model is interesting to test our kernel approximations since we are able to sample exactly both the model and its approximation, without any additional discretization error. Therefore, the only bias comes from our approximation. 
 We thus consider a two dimensional Brownian motion~$W$ and the following dynamics:
\begin{align}
  S_t&=S_0\exp \left( \int_0^t \sqrt{\nu_s}(\rho dW^1_s + \sqrt{1-\rho^2} dW^2_s ) - \frac12 \int_0^t \nu_s ds \right),\\
  \nu_t&=\nu_0 \exp \left(\eta \sqrt{2H} \int_0^t (t-s)^{H-1/2} dW^1_s -\frac{\eta^2}{2}t^{2H}\right).
\end{align}

We first describe the algorithm of Bayer et al.~\cite{BaFrGa}. It consists in discretizing the time interval $[0,T]$ with $N$ time steps. Thus, one has to simulate the Gaussian vector $(\int_0^{\frac lN T} (\frac lN T-s)^{H-1/2} dW^1_s, W^1_{\frac lN T})_{l=1,\dots,N}$ by computing a Cholesky decomposition of the covariance matrix. Then, the values of $\nu_{\frac lN T}$ are sampled exactly, and one approximate $S$ with the following scheme, for $l\in\{1,\dots,N\}$:
$$\hat{S}_{\frac lN T}=\hat{S}_{\frac{l-1}N T}\exp \left( \nu_{\frac{l-1}N T} \left(\rho(W^1_{\frac lN T}-W^1_{\frac{l-1}N T})+\sqrt{1-\rho^2}(W^2_{\frac lN T}-W^2_{\frac{l-1}N T})\right) - \frac 12 \nu_{\frac{l-1}N T} \frac T N \right).$$

      Here, we furthermore approximate $\nu$ by using an approximation of the rough kernel. Namely we use that
      \begin{align*}\sqrt{2H} \int_0^t (t-s)^{H-1/2} dW^1_s &= \sqrt{2H}\Gamma(H+1/2) \int_0^tG_{\lambda_H}(t-s)dW^1_s\\ &\approx \sqrt{2H}\Gamma(H+1/2) \int_0^t \hat{G}^{sys}_n(t-s)dW^1_s.
      \end{align*}
      Since the approximation is a combination of exponential functions, we can simulate it exactly by the Gaussian vector $(\int_0^{\frac lN T} \exp \left(-\rho_i(\frac lN T-s) \right) dW^1_s, W^1_{\frac lN T})_{l\in\{1,\dots,N\}, i}$ again by computing a Cholesky decomposition of the covariance matrix. Then, we define the following approximation of~$\nu$ with $\bar{c}=\eta \sqrt{2H}\Gamma(H+1/2)$:
      \begin{equation}
        \hat{\nu}_{\frac lN T }=\nu_{\frac {l-1}N T} \exp \left( \bar{c} \int_0^t \hat{G}^{sys}_n(t-s)dW^1_s - \frac 12 \bar{c}^2 \int_0^t \hat{G}^{sys}_n(t-s)^2ds\right).
      \end{equation}
      Note that the integral $ \int_0^t \hat{G}^{sys}_n(t-s)^2ds$ can be easily  calculated exactly. We notice that it is important in numerical applications to compute it instead of using $\frac{\eta^2}2[ (\frac lN T)^{2H}-(\frac {l-1}N T)^{2H}]$ that introduces some bias. This slight modification improves significantly the numerical results in approximating the smile curve.

\begin{figure}[H]
    \centering
  \includegraphics[width=0.8\linewidth]{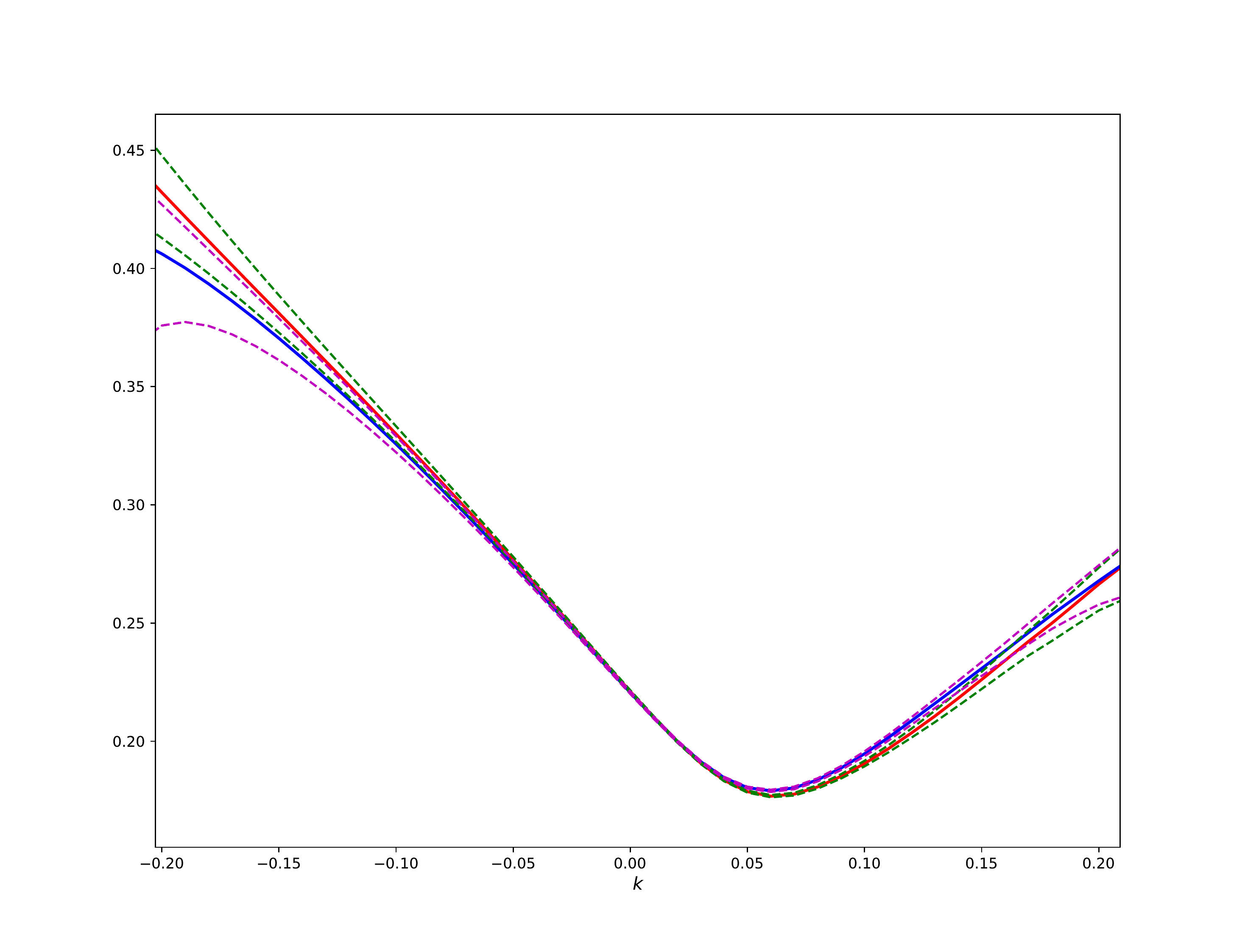} 
  \caption{Implicit volatility of the Call option with strike $e^k$ obtained by the Monte-Carlo estimator: the method of Bayer et al~\cite{BaFrGa} in blue, our proposed approximation in red. Respective 95\% confidence intervals delimited with dotted lines in magenta and green. Parameters: $H=0.07$, $S_0=1$, $v_0=0.235^2$, $\eta=1.9$, $\rho=-0.9$, $T=0.041$, $N=20$ and the approximation kernel $\hat{G}^{sys}_n(t)$ with $n=20$.}
\label{figure_smiles}
  \end{figure}
      
In Figure~\ref{figure_smiles}, we have plotted the smile obtained by the algorithm of Bayer et al.~\cite{BaFrGa}. We have taken back the parameter sets of this paper (also taken in Bennedsen et al.~\cite{BeLuPa}) and we focus on their most challenging example, i.e. the one with short maturity $T=0.041$. We have approximated by Monte-Carlo the value of $\E[(S_T-e^{k})^+]$ with $10^6$ samples. The approximation that we propose is very close to the smile produced by the method proposed in~\cite{BaFrGa}, which shows its relevance. Note that on this specific example, there is no particular advantage to use our kernel approximation rather than the one of Bayer et al.~\cite{BaFrGa} since everything can be sampled exactly. However, if one uses for the volatility a more involved Volterra SDE with the rough kernel, exact sampling is no longer possible while our kernel approximations can still be used since they correspond to a classical SDE in a higher dimension. This is the purpose of the next subsection. 

  \subsection{Comparison between different numerical schemes for the rough Heston model}\label{subsec_rHeston}

We now focus on the more challenging case of the rough Heston model introduced by~\cite{EleRos}. This model has the following dynamics
\begin{equation} \label{eq:RoughHeston}
	\begin{split}
		S_t&=S_0+\int_0^t S_s\, \sqrt{V_s}~d \big( \rho\, W_s+\sqrt{1-\rho^2}\, W_s^{\bot} \big),\\
		V_t&=V_0+\int_0^t G_{\lambda_H}(t-s) \Big( \big( \theta-\lambda V_s \big)~ds + \sigma \sqrt{V_s}~dW_s \Big),
	\end{split}
	\end{equation}
where $(W, W^{\bot})$ are two independent Brownian motions, $G_{\lambda_H}(t)$ is the rough kernel function~\eqref{def_rough_kernel}, $S_0,V_0,\theta,\lambda,\sigma > 0$ and $\rho\in [-1,1]$. To approximate this process, it is more convenient to work with $Y_t=\log(S_t)$.

\subsubsection{Presentation of the Volterra Euler scheme, the multifactor and the hybrid multifactor Euler schemes}\label{subsec_compar}

Richard et al.~\cite{RTY,RTY2} have studied the Volterra Euler scheme for general SVE with Lipschitz coefficients. For~\eqref{eq:RoughHeston}, they consider the following scheme on the time grid $t_k=kT/N$, $k\in\{0,\dots N-1\}$:
\begin{equation} \label{eq:Euler RoughHeston}
	\begin{split}
		Y^N_{t_{k+1}}&=Y^N_{t_{k}}- \frac 12 (V^N_{t_k})_+ \frac TN +\sqrt{ (V^N_{t_k})_+} \left(\rho (W_{t_{k+1}}-W_{t_k})+\sqrt{1-\rho^2} (W^{\bot}_{t_{k+1}}-W^{\bot}_{t_k}) \right) ,\\
		V^N_{t_{k+1}}&=V_0+\sum_{j=0}^k G_{\lambda_H}\left((k+1-j)\frac TN \right) \Big( \big( \theta-\lambda (V^N_{t_j})_+ \big)\frac{T}{N} + \sigma \sqrt{ (V^N_{t_j})_+} (W_{t_{j+1}}-W_{t_j}) \Big). 
	\end{split}
	\end{equation}
Due to summation in the definition of $V^N_{t_k+1}$, the computational complexity is of order of $N^2$.

We consider an approximating kernel $\hat{G}(t)=\sum_{i=1}^n\alpha_i e^{-\rho_i t}$ with $\alpha_i>0$ and $0\le \rho_1<\dots<\rho_n$. We now write the multifactor Euler scheme corresponding to the multidimensional SDE approximation~\eqref{SDE_approx_1}.  For this aim, we can use for $V$ the multifactor Euler scheme given by~\eqref{def_Euler_Scheme_SDE}, which  leads to the following numerical scheme:
\begin{equation} \label{eq:Euler RoughHeston_SDE}
	\begin{split}
		\hat{Y}^N_{t_{k+1}}&=\hat{Y}^N_{t_{k}}- \frac 12 (\hat{V}^N_{t_k})_+ \frac TN +\sqrt{ (\hat{V}^N_{t_k})_+} \left(\rho (W_{t_{k+1}}-W_{t_k})+\sqrt{1-\rho^2} (W^{\bot}_{t_{k+1}}-W^{\bot}_{t_k}) \right) ,\\
		 \hat{V}^{i,N}_{t_{k+1}}&=e^{-\rho_i \frac TN}\left( \hat{V}^{i,N}_{t_{k}}+ (\theta-\lambda (\hat{V}^{N}_{t_k})_+)\frac TN +\sigma \sqrt{(\hat{V}^{N}_{t_k})_+} (W_{t_{k+1}}-W_{t_k})\right), \ 1\le i\le n \\
                \hat{V}^N_{t_{k+1}}&=V_0+\sum_{i=1}^n \alpha_i \hat{V}^{i,N}_{t_{k+1}},
	\end{split}
\end{equation}
with $\hat{V}^{i,N}_{t_{0}}=0$. Note that by Theorem~\ref{prop_euler}, $\hat{V}^N_{t_{k+1}}$ satifies the same recurrence formula as $V^N_{t_{k+1}}$ in~\eqref{eq:Euler RoughHeston}, when replacing $G_{\lambda_H}$ by $\hat{G}$.  
Unlike the scheme~\eqref{eq:Euler RoughHeston}, this scheme has a computational complexity of order $n\times N$. This is a clear advantage of our scheme when $N$ gets large. Besides,  we can reduce $n$ as follows by using the idea of Corollary~\ref{lem_ntilde}, even if the diffusion coefficient is not Lipschitz. More precisely, we define (we take $\beta=1$ in~\eqref{def_ntilde} since it leads to accurate results)
\begin{equation}\label{def_ntilde_2}
  \tn=\inf \{ k \in \{1,\dots,n\} :  \sum_{i=k+1}^n \alpha_i e^{-\rho_i \frac T N} \le \frac TN  \text{ for } i\ge k+1 \}.
\end{equation}
Then, we simply consider
\begin{equation} \label{eq:Euler RoughHeston_SDE_bis}
	\begin{split}
		\tilde{Y}^N_{t_{k+1}}&=\tilde{Y}^N_{t_{k}}- \frac 12 (\tilde{V}^N_{t_k})_+ \frac TN +\sqrt{ (\tilde{V}^N_{t_k})_+} \left(\rho (W_{t_{k+1}}-W_{t_k})+\sqrt{1-\rho^2} (W^{\bot}_{t_{k+1}}-W^{\bot}_{t_k}) \right) ,\\
		 \tilde{V}^{i,N}_{t_{k+1}}&=e^{-\rho_i \frac TN}\left( \tilde{V}^{i,N}_{t_{k}}+ (\theta-\lambda (\tilde{V}^{N}_{t_k})_+)\frac TN +\sigma \sqrt{(\tilde{V}^{N}_{t_k})_+} (W_{t_{k+1}}-W_{t_k})\right),  \ 1\le i\le \tn \\
                \tilde{V}^N_{t_{k+1}}&=V_0+\sum_{i=1}^{\tn} \alpha_i \tilde{V}^{i,N}_{t_{k+1}}.
	\end{split}
	\end{equation}

Very recently, an hybrid multifactor scheme has been proposed by R\o mer~\cite{Romer}  that combines the hybrid approximation proposed by Bennedsen et al.~\cite{BeLuPa} for general kernels and the well-known multifactor approximation of  completely monotone kernels. The principle of this scheme is to approximate $V_{t_{k+1-\kappa}}$ for some $\kappa \in \N^*$ by using the multifactor approximation (denoted here by $\check{V}^{multi}_{t_{k+1-\kappa}}$) and then approximate $V_{t_{k+1}}$ by
\begin{align*}
 & V_{t_{k+1}}= V_{t_{k+1-\kappa}}+\int_{t_{k+1-\kappa}}^{t_{k+1}}G_{\lambda_H}(t_{k+1}-s)[(\theta-\lambda V_s)ds+\sigma \sqrt{V_s}dW_s] \\
  &\approx \check{V}^{multi}_{t_{k+1-\kappa}} +\sum_{i=1}^\kappa (\theta-\lambda V_{t_{k+1-i}}) \int_{t_{k+1-i}}^{t_{k+2-i}}G_{\lambda_H}(t_{k+1}-s)ds + \sigma \sqrt{ V_{t_{k+1-i}}}\int_{t_{k+1-i}}^{t_{k+2-i}}G_{\lambda_H}(t_{k+1}-s)dW_s.
\end{align*}
As noticed by Bennedsen et al~\cite{BeLuPa} and then by R\o mer~\cite{Romer}, the choice $\kappa=1$ is usually sufficient in practice, which leads to the following scheme~\cite[Definition 1]{Romer} for $k\in\{0,\dots,N-1\}$:
\begin{align} 
		\check{Y}^N_{t_{k+1}}&=\check{Y}^N_{t_{k}}- \frac 12 (\check{V}^N_{t_k})_+ \frac TN +\sqrt{ (\check{V}^N_{t_k})_+} \left(\rho (W_{t_{k+1}}-W_{t_k})+\sqrt{1-\rho^2} (W^{\bot}_{t_{k+1}}-W^{\bot}_{t_k}) \right) , \notag\\
		 \check{V}^{i,N}_{t_{k+1}}&=\frac 1{1+\rho_i \frac TN} \left( \check{V}^{i,N}_{t_{k}}+ (\theta-\lambda (\check{V}^{N}_{t_k})_+)\frac TN +\sigma \sqrt{(\check{V}^{N}_{t_k})_+} (W_{t_{k+1}}-W_{t_k})\right), \ 1\le i\le n \notag\\
                 \check{V}^N_{t_{k+1}}&=\check{V}^{multi}_{t_{k}}  + (\theta-\lambda(\check{V}^N_{t_{k}})_+ ) \int_0^{T/N}G_{\lambda_H}(s)ds   + \sigma \sqrt{(\check{V}^N_{t_{k}})_+}\int_{t_{k}}^{t_{k+1}}G_{\lambda_H}(t_{k+1}-s)dW_s \notag\\
             \check{V}^{multi}_{t_{k}}    &=V_0+\sum_{i=1}^n \alpha_i e^{-\rho_i \frac TN} \check{V}^{i,N}_{t_{k}}. \label{eq:Euler RoughHeston_hybrid}
\end{align}
One needs to sample exactly  the Gaussian vector $\left(W_{t_{k+1}}-W_{t_k} ,\int_{t_{k}}^{t_{k+1}}G_{\lambda_H}(t_{k+1}-s)dW_s\right)$ which has an explicit covariance matrix.  Note that this hybrid multifactor scheme has similarities with~\eqref{eq:Euler RoughHeston_SDE}. The exponential factor  is replaced by $\frac 1{1+\rho_i \frac TN}$, which does not change that much in practice. 
Thus, the main difference between the two schemes is the approximation on the last step.

\subsubsection{Numerical Results}\label{subsec_num_schemes}

To test these schemes, we have taken back the numerical experiments of Richard et al.~\cite{RTY,RTY2} with the following parameters: $V_0=\theta=0.02$, $\lambda=0.3$, $\sigma=0.3$, $\rho=-0.7$, $S_0=1$, $H=0.1$. We first compute the European call price $\E[(S_T-K)_+]$ with strike $K=1$, maturity $T=1$ and zero interest rates. The approximated exact value of this option computed using Fourier pricing techniques is $0.05683$. We use the approximating kernel $\hat{G}^{sys}_n$ given by the systematic approach~\eqref{def_sys} and then the selection of the first $\tilde{n}$ components given by~\eqref{def_ntilde_2}. The corresponding $L^2$ error is $\sqrt{\int_0^1(\hat{G}^{sys}_n(t)-G_{\lambda_H}(t))^2dt }\approx 0.01523$ and the discrete $L^2$ errors are
{\small \begin{align}
  \sqrt{\frac T N \sum_{k=1}^N (\hat{G}^{sys}_n(kT/N)-G_{\lambda_H}(kT/N))^2}&\approx 0.00783628, \label{val_num_approx} \\ \sqrt{\frac T N \sum_{k=1}^N (\sum_{i=1}^{\tilde{n}} \alpha_i e^{-\rho_ikT/N}-G_{\lambda_H}(kT/N))^2} &\approx 0.00783633, \notag
  \end{align}
}

\noindent for $N=160$ and $\tilde{n}=55$. This indicates in view of Corollary~\ref{cor_euler} why the prices obtained with the Euler scheme and the multifactor Euler scheme are very close in Tables~\ref{table_1},~\ref{table_3}. Also, in view of Theorem~\ref{thm_non_asymp_scheme}, the difference between the Euler schemes~$\hat{Y}^N$ and $\tilde{Y}^N$ is very small, while the time complexity are respectively proportional to $n\times N$ and $\tilde{n} \times N$, which is a clear gain of  the acceleration procedure given by~\eqref{def_ntilde_2} in view of Corollaries~\ref{cor_euler} and~\ref{lem_ntilde}. For the different values of $N$ and $n=100$, we typically have values of $\tilde{n}$ between $50$ and $65$.  For a fair comparison between our multifactor Euler scheme and the hybrid multifactor scheme, we have used the same  $\alpha$'s and $\rho$'s coming from $\hat{G}^{sys}_n$ given by the systematic approach~\eqref{def_sys} with $n=100$ (i.e. $n=100$ values of $\alpha$ and  $\rho$), and combined with the selection of the $\tilde{n}$ first components given by~\eqref{def_ntilde_2}. 

{\scriptsize
 \begin{table}
   \begin{tabular}{|r||l|l|l||l|l|l||l|l|l|}
     \hline
     $N$& Mean  & 95\% prec. & Time (s) & Mean   & 95\% prec. & Time (s)& Mean   & 95\% prec. & Time (s)  \\
     \hline
     10 &0.05922   & 1.5e-4   & 10   & 0.05919   &  1.5e-4    &  3  & 0.06767 & 1.8e-4 & 12 \\
     20& 0.05883  & 1.5e-4   & 32  & 0.05868    &  1.5e-4  & 13 & 0.06619 & 1.7e-4 & 36 \\
     40& 0.05848  & 1.4e-4   & 67 & 0.05845  & 1.4e-4   & 50 & 0.06471  & 1.6e-4 & 73 \\
     80 & 0.05821 & 1.4e-4  & 134  & 0.05814  & 1.4e-4   & 198 & 0.06337 & 1.6e-4 & 144  \\
     160 & 0.05801 & 1.4e-4   & 274 & 0.05780  & 1.4e-4   & 745 &0.06225 & 1.5e-4 & 300 \\
     320 & 0.05777 & 1.4e-4 & 583 & 0.05783 & 1.4e-4  & 3136 & 0.06135 & 1.5e-4& 614 \\ 
     \hline
   \end{tabular}
    \vspace{0.1cm}
   \caption{Price of the European call option in the rough Heston model by using the multifactor Euler scheme (left), the Volterra Euler scheme (middle) and the hybrid multifactor scheme (right). }\label{table_1}
 \end{table}

  \begin{table}
   \begin{tabular}{|r||l|l|l||l|l|l||l|l|l|}
     \hline
     $N$& Mean  & 95\% prec. & Time (s) & Mean   & 95\% prec. & Time (s)& Mean   & 95\% prec. & Time (s)   \\
     \hline
     10 & 0.08134  & 1.5e-4 & 11 & 0.08153  & 1.5e-4  & 3 &  0.09010 & 1.6e-4 & 12\\
     20& 0.08563&   1.4e-4 & 32  &0.08559  &  1.4e-4 & 13 & 0.09320 & 1.5e-4 & 35 \\
     40& 0.08835&  1.4e-4  & 68 & 0.08861  &  1.4e-4 & 51 & 0.09543 & 1.6e-4 & 73 \\
     80 &  0.09047 &  1.4e-4 & 136 & 0.09069  &  1.4e-4 & 199 & 0.09693 & 1.5e-4 & 152\\
     160 & 0.09193 &  1.4e-4 & 279  & 0.09204 &  1.4e-4 & 743 & 0.09766 & 1.5e-4 & 291 \\
     320 & 0.09294 &  1.4e-4 & 561  & 0.09310  &  1.4e-4 & 3143 & 0.09778 & 1.5e-4 & 636\\
     \hline
   \end{tabular}
   \vspace{0.1cm}
   \caption{Price of the Lookback call option in the rough Heston model by using the multifactor Euler scheme (left), the Volterra Euler scheme (middle) and the hybrid multifactor scheme (right).}\label{table_3}
  \end{table}

}

We have indicated in Table~\ref{table_1}, for the three schemes, the value of the Monte-Carlo estimator with a sample of size~$10^6$ with the corresponding precision (half-with of the 95\% confidence interval)  and computation time. As expected from Theorem~\ref{prop_euler}, the values obtained by the Volterra Euler scheme and the multifactor Euler scheme are quite close. However, as noticed in our complexity analysis of both schemes, the smaller is the time step, the greater is the gain in favour of the multifactor Euler scheme. The hybrid multifactor scheme has a slightly higher computational cost with respect to the multifactor Euler scheme, which is due to the sampling of the Gaussian vector. The approximation induced by this latter random vector leads to a larger bias with respect to the other methods. This may be explained by the large variance of $\int_{t_{k}}^{t_{k+1}}G_{\lambda_H}(t_{k+1}-s)dW_s$  compared to the one of $W_{t_{k+1}}-W_{t_k}$ and by the convexity of the payoff. The hybrid multifactor scheme leads then to a higher price than the one of the multifactor scheme, which is already above the theoretical price. The slow convergence of the hybrid multifactor scheme is also noticed by~\cite[Figure 9  and comments below]{Romer}.

We then compute the Lookback call option prices $\E[(\max_{t\in[0,T]}S_t-K)_+]$ with the maximum approximated by $\max_{0\le k\le N}S_{t_k}$. We have indicated in Table~\ref{table_3}  the values of the Monte-Carlo estimators with a sample of size~$10^6$ with the corresponding precision and computation time. We notice again that the values obtained by the Volterra Euler scheme and the multifactor Euler scheme are close, and that the gain in computation time provided by the multifactor approximation gets more and more significant as the time step decreases. Contrary to European option case, we do not have a reference price for the Lookback option given by a semi explicit formula. Thus, we cannot say between the multifactor and the hybrid multifactor which one produces the lowest bias on this example.


\subsubsection{Alternative Euler scheme on the integrated volatility process}

Another way to simulate this process has been proposed by Richard et al.~\cite{RTY2}. It is based on  an alternative  writing of the rough Heston model based on the integrated volatility process $X$  with $X_t=\int_0^tV_sds$  has been proposed by Abi Jaber~\cite{abi2021weak} 
\begin{align}\label{eq:IntegVarHeston}
	\begin{split}
		S_t&=S_0+\int_0^t S_s~d \big(\rho M_s+\sqrt{1-\rho^2}M_s^{\bot} \big),\\
		X_t&=V_0 t +\int_0^t G_{\lambda_H}(t-s)\big(\theta s -\lambda X_s+\sigma M_s\big)~ds,
	\end{split}
	\end{align}
where $M,~M^{\bot}$ are two orthogonal continuous martingales with quadratic variation $\langle M \rangle  = \langle M^{\bot}\rangle =X$. Then, $Y_t=\log S_t$ satisfies 	$$ Y_t = Y_0 - \frac12 X_t +  \rho M_t + \sqrt{1-\rho^2}M^{\bot}_t. 		$$
Richard et al.~\cite{RTY2} propose an alternative discretization scheme based on approximating the martingales  $(M, M^{\bot})$: 
\begin{align}
   X^N_{t_{k+1}}&= V_0 t_{k+1} + \sum_{j=0}^{k} G_{\lambda_H}(t_{k+1}-t_j)\left( \theta t_j -\lambda \overline{X}^N_{t_j}+\sigma M^N_{t_j}\right),  k\in\{0,\dots,N-1\},\label{Euler_MAX}\\
  Y^N_{t_k}&=Y_0  - \frac12 \overline{X}^N_{t_k} +  \rho M^N_{t_k} + \sqrt{1-\rho^2}M^{N,\bot}_{t_k}, k\in\{1,\dots,N\}, \notag\\
M^N_{t_k}&= \sum_{j=1}^k \sqrt{\overline{X}^N_{t_j}-\overline{X}^N_{t_{j-1}}}Z_j,
		M^{N,\bot}_{t_k} = \sum_{j=1}^k \sqrt{\overline{X}^N_{t_j}-\overline{X}^N_{t_{j-1}}}Z_j^{\bot}, k\in\{1,\dots,N\},\notag
\end{align}
with $\overline{X}^N_{t_j} := \underset{ 0\leq l\leq j }{\max} X^N_{t_l}$, $X^N_{t_0}=M^N_{t_0}=M^{N,\bot}_{t_0}=0$, and $(Z_j,Z_j^{\bot})_{j \ge 1}$ is a sequence of  i.i.d. random variables with standard Gaussian distribution $\mathcal{N}(0,1)$. They also prove in~\cite[Theorem 2.3]{RTY2} that the discretization scheme $(S^N, X^N)$ weakly converges to $(S, X)$.

Again, we consider an approximating kernel $\hat{G}(t)=\sum_{i=1}^n\alpha_i e^{-\rho_i t}$ with $\alpha_i>0$ and $0\le \rho_1<\dots<\rho_n$, and write the multifactor Euler scheme associated to~$X$, which leads to the following scheme:
\begin{align} 
  	 \hat{X}^{i,N}_{t_{k+1}}&=e^{-\rho_i \frac TN}\left( \hat{X}^{i,N}_{t_{k}}+ \big(\theta t_{k}-\lambda (\hat{X}^{N}_{t_{k}})_+ +\sigma \hat{M}^N_{t_{k}}\big)\frac TN \right), k\in \{0,\dots,N-1\} \notag\\
         \hat{X}^N_{t_k}&= V_0 t_k + \sum_{i=1}^{n} \alpha_i \hat{X}^{i,N}_{t_{k}}, k\in \{1,\dots,N\}, \label{Euler_MAX_SDE}\\
           \hat{Y}^N_{t_k}&=Y_0  - \frac12 \overline{\hat{X}}^N_{t_k} +  \rho \hat{M}^N_{t_k} + \sqrt{1-\rho^2}\hat{M}^{N,\bot}_{t_k}, k\in \{1,\dots,N\}, \notag\\
\hat{M}^N_{t_k}&= \sum_{j=1}^k \sqrt{\overline{\hat{X}}^N_{t_j}-\overline{\hat{X}}^N_{t_{j-1}}}Z_j,
		\hat{M}^{N,\bot}_{t_k} = \sum_{j=1}^k \sqrt{\overline{\hat{X}}^N_{t_j}-\overline{\hat{X}}^N_{t_{j-1}}}Z_j^{\bot}, k\in \{1,\dots,N\},\notag
\end{align}
with $\overline{\hat{X}}^N_{t_j} := \underset{ 0\leq l \leq j }{\max} \hat{X}^N_{t_l}$, $\hat{X}^{i,N}_{t_{0}}=\hat{M}^{N}_{t_0}=\hat{M}^{N,\bot}_{t_0}=0$. By Theorem~\ref{prop_euler}, $\hat{X}^{i,N}_{t_{k}}$ satisfies the recurrence formula~\eqref{Euler_MAX} replacing $G_{\lambda_H}$ by~$\hat{G}$. Besides, we can use the same idea as in Corollary~\ref{lem_ntilde} to reduce the dimension of this approximation: we thus build the approximation $(\tilde{X}^N,\tilde{Y}^N)$ associated to~\eqref{Euler_MAX_SDE}  which we have used in the next numerical experiments. 

We have indicated in Table~\ref{table_2} (resp. Table~\ref{table_4}) the values of the Monte-Carlo estimators associated to the schemes~\eqref{Euler_MAX} and~\eqref{Euler_MAX_SDE} for the European (resp. Lookback) Call option. We have taken the same parameters and the same approximating kernel as in Subsection~\ref{subsec_compar}.  We observe that the Volterra Euler scheme on the integrated volatility~\eqref{Euler_MAX} and the corresponding multifactor Euler scheme~\eqref{Euler_MAX_SDE} give very similar values. In particular, we get back the observation of Richard et al.~\cite{RTY2} that the scheme on the integrated volatility gives a lower bias than the  scheme on the volatility for the European option, but yields instead to a larger bias for the lookback option.  Again, the computation time required by the multifactor Euler scheme is much lower as the time step gets smaller, which shows the relevance of Scheme~\eqref{Euler_MAX_SDE}, and more generally the relevance of using the multifactor approximation of SVE with kernels of completely monotone type provided that we have an accurate approximation of the kernels. 
  \begin{table}[H]
    \begin{tabular}{|r||l|l|l||l|l|l|}
      \hline
      $N$& Mean  & 95\% prec. & Time (s) & Mean   & 95\% prec. & Time (s) \\
     \hline
     10 & 0.05791  & 1.5e-4  &12 & 0.05802 & 1.5e-4 & 3 \\
     20 & 0.05740  & 1.4e-4 & 41 & 0.05747 & 1.4e-4 & 13 \\
    40 & 0.05720  & 1.4e-4  & 88 & 0.05715 & 1.4e-4 & 50 \\
    80 & 0.05698  & 1.4e-4  & 187&  0.05689& 1.4e-4 & 196 \\
    160 & 0.05696 & 1.4e-4  & 408 &  0.05688  & 1.4e-4 & 767  \\
     \hline
    \end{tabular}
     \vspace{0.1cm}
   \caption{Price of the European call option in the rough Heston model by using the multifactor Euler scheme on the integrated volatility (left) and the Volterra Euler scheme on the integrated volatility (right).}\label{table_2}
    \end{table}

  \begin{table}[H]
   \begin{tabular}{|r||l|l|l||l|l|l|}
     \hline
     $N$& Mean. Val.  & 95\% prec. & Comp. time. & Mean. Val.  & 95\% prec. & Comp. time.  \\
     \hline
     10 &0.07784 &  1.4e-4 & 12 & 0.07765 &  1.4e-4 & 3 \\
     20 &0.08180 &  1.4e-4& 40  & 0.08186 &  1.4e-4 & 13\\
    40 & 0.08511 &  1.4e-4& 88  & 0.08510 &  1.4e-4 & 49  \\
    80 & 0.08770 &  1.4e-4& 189 & 0.08783 &  1.4e-4 & 194 \\
    160 &0.08964 &  1.4e-4& 402 & 0.08958 &  1.4e-4 & 775 \\
    320 & 0.09089&  1.4e-4& 831 & 0.09100 &  1.4e-4 & 3130 \\
     \hline
   \end{tabular}
   \vspace{0.1cm}
   \caption{Price of the Lookback call option in the rough Heston model by using the multifactor Euler scheme on the integrated volatility (left) and the Volterra Euler scheme on the integrated volatility (right).}\label{table_4}
  \end{table}

\bibliographystyle{abbrv}
\bibliography{Rough_CIR}
\end{document}